\documentclass[12pt]{amsart}
\usepackage[margin=1.0in]{geometry}
 
\usepackage{comment}
\usepackage{latexsym, amsmath, amssymb,amsthm,amsopn,amsfonts}
\usepackage{version}
\usepackage{epsfig,graphics,color,graphicx,graphpap}
\usepackage{amssymb}
\usepackage[hypertexnames=false]{hyperref}
\usepackage{todonotes}
\usepackage[normalem]{ulem}
\usepackage{mathrsfs}
\newcommand{\redsout}{\bgroup\markoverwith{\textcolor{red}{\rule[0.5ex]{2pt}{.4pt}}}\ULon}
\newcommand{\e}{\varepsilon}

\newcommand{\p}{\partial}

\newcommand{\R}{\mathbb{R}}

\newcommand{\m}{\mathbf{m}}
\newcommand{\M}{\mathbf{M}}

\newcommand{\logep}{\left| \log \e \right|}

\newcommand{\supp}{\text{supp}}
\newcommand{\calS}{\mathcal{S}}

\newcommand{\LN}{\left\|}
\newcommand{\RN}{\right\|}

\newcommand{\LC}{\left(}
\newcommand{\RC}{\right)}
\newcommand{\LB}{\left[}
\newcommand{\RB}{\right]}

\newcommand{\eps}{\varepsilon}
\newcommand{\ov}{\overline}
\newcommand{\eqal}[1]{\begin{equation}\begin{aligned}#1\end{aligned}\end{equation}}

\newcommand{\re}{\mathbb R}
\newcommand{\pd}{\,\partial}
\newcommand{\mc}{\mathcal}

\newcommand{\Om}{\Omega}

\newcommand{\al}{\alpha}
\newcommand{\wh}{\widehat}
\newcommand{\wt}{\widetilde}
\newcommand{\ts}{\wt{\mc S}}
\newcommand{\s}{\mc S}
\newcommand{\sg}{\sigma}
\newcommand{\ms}{\mathscr}
\newcommand{\vp}{\varphi}

 \definecolor{skyblue}{rgb}{0.85,0.85,1}

\newtheorem{theorem}{Theorem}

\newtheorem{lemma}{Lemma}
\newtheorem{definition}{Definition}

\theoremstyle{remark}

\begin{document}

\title[An inverse problem from condensed matter physics]{An inverse problem \\ from condensed matter physics} 

\author[Lai]
{Ru-Yu Lai}
\address{School of Mathematics, University of Minnesota, Minneapolis, Minnesota 55455
}
\email{rylai@umn.edu}

\author[Shankar]
{Ravi Shankar}
\address{Department of Mathematics, University of Washington, Seattle, WA 98195-4350, USA}
\email{shankarr@uw.edu}

\author[Spirn]
{Daniel Spirn}
 \address{School of Mathematics, University of Minnesota, Minneapolis, Minnesota 55455
 }
\email{spirn@math.umn.edu}

\author[Uhlmann]
{Gunther Uhlmann}
\address{Department of Mathematics, University of Washington, Seattle, WA 98195-4350, USA;}\address{  Department of Mathematics, University of Helsinki, Helsinki, Finland;}\address{HKUST Jockey Club Institute for Advanced Study, HKUST, Clear Water Bay, Kowloon, Hong Kong.}
\email{gunther@math.washington.edu}

\thanks{R.-Y. L. was partly supported by the AMS-Simons Travel Grants.
R.S. and D.S. were supported in part by the NSF. G.U. was partly supported by the NSF, a Si-Yuan Professorship at IAS, HKUST, and a FiDiPro at the University of Helsinki.}

\begin{abstract}
We consider the problem of reconstructing the features of a weak anisotropic background potential by the trajectories of vortex dipoles in a nonlinear Gross-Pitaevskii equation.  At leading order, the dynamics of vortex dipoles are given by a Hamiltonian system.  If the background potential is sufficiently smooth and flat, the background can be reconstructed using ideas from the boundary and the lens rigidity problems. We prove that reconstructions are unique, derive an approximate reconstruction formula, and present numerical examples.
\end{abstract}

\maketitle

\section{Introduction}
 
We consider the question of reconstruction of a potential from the dynamical behavior 
of vortex dipoles in an inhomogeneous Gross-Pitaevskii equation, 
\begin{equation} \label{eq:igl}
i \p_t u_\e = \Delta u_\e + {1\over \e^2} \LC \mathcal{K}^2_\e(x)  - |u_\e|^2 \RC u_\e
\end{equation}
on $\mathbb{R}^2$.   
The inhomogenous Gross-Pitaevskii equation \eqref{eq:igl} arises in many places, including Bose-Einstein Condensation (BEC), nonlinear optics \cite{Arecchi,arecchi1991vortices, PR}, and the  behavior of  superfluid ${}^4$He near a boundary  \cite{schwarz1985three, tsubota1993pinning, photograph}.   
We refer to \cite{PR} for discussion of the relevance to physical problems.  Our motivation stems from the nontrivial dynamical behavior of vortices in Bose-Einstein Condensates, as discussed below. 

Vortices are a fundamental feature of  Bose-Einstein condensates.  They are localized regions where the condensed matter loses its superstate and where the modulus of $u_\e$ locally vanishes.  Vortices also carry a quantized degree about each nodal point, and so in two dimensions  solutions of \eqref{eq:igl} with $\e \ll 1$ take the form  
\begin{equation} \label{eq:vortexform}
u_\e(x,t) \approx \prod_{j=1}^d Q_\e(|x-a_j(t)|) \LC {x - a_j(t) \over |x - a_j(t)|} \RC^{d_j},
\end{equation}
where $d_j \in \{-1,1\}$ is the degree of the vortex about $a_j(t)$ and $ Q_\e (s) \to 0$ as $s \to 0$ and roughly $|u_\e(x)|^2 \approx \mathcal{K}_\varepsilon^2(x)$ in $\cap_j\{ |x - a_j| \gg \e\} $. 

Dynamical motion laws for the vortex positions, $a_j(t)$ in \eqref{eq:igl} were first proposed by Fetter \cite{fetter1} for trivial backgrounds $\mathcal{K}_\varepsilon (x) \equiv 1$, and it was shown that the $a_j(t)$ with initial positions $a_j^0$ are governed by an ODE that depends solely on vortex-vortex interaction,
\begin{equation} \label{eq:ODEsf}
\pi d_j \dot{a}_j = \nabla^\perp_{a_j} W(a,d),  
\end{equation}
where $a = (a_1, \ldots, a_n)$, $d = \{ d_1, \ldots, d_n\}$, and $W(a,d)$ is the Coulomb potential, 
\begin{equation} \label{eq:coulomb}
W(a,d) = - \sum_{j \neq k} d_j d_k \log |a_j - a_k| .
\end{equation}
Here and in what follows, we will always take $d_j \in \{ \pm 1\}$.  
%We first introduce some notations.
For $u=(u_1, u_2)\in \R^2$, we define $u^\perp = (u_2, -u_1)$.
If a function $w : \R^2 \rightarrow \R$, then we denote the gradient of $w$ by $\nabla  w = (\partial_1 w, \partial_2 w)$ and also
$\nabla^\perp w = (\partial_2 w, -\partial_1 w)$.
Similarly, if a function $f: \R^{2n}\rightarrow \R$. We define $\nabla_{a_i} f$, the gradient of $f$ with respect to $a_i = (a_{i,1}, a_{i,2})$, by 
$
   \nabla_{a_i} f = (\partial_{a_{i,1}}f, \partial_{a_{i,2}}f) 
$
and also define the function
$\nabla^\perp_{a_i} f$ by 
$
   \nabla^\perp_{a_i} f = (\partial_{a_{i,2}} f, -\partial_{a_{i,1}} f).
$

The ODE system \eqref{eq:ODEsf} is identical to the Kirchoff-Onsager law for point vortices in two dimensional incompressible Euler equations.  
The connection between \eqref{eq:igl} and \eqref{eq:ODEsf} was rigorously proved by \cite{CJ}, and also \cite{CJ99, LX}. We also note the numerical works \cite{Bao3, Bao1, Bao2} that study the associated vortex dynamics.
When the BEC has a strong trapping potential, $ \mathcal{K} _\e(x) = O(1)$, then vortices tend to travel slowly along the level sets of the Thomas-Fermi profile.    Fetter-Svidzinsky \cite{fetter2} wrote down the corresponding dynamical law,
 \begin{equation} \label{eq:ODEbec}
\pi d_j \dot{a}_j = \nabla^\perp_{a_j} \log \rho_{TF}(a_j),
\end{equation}
where the Thomas-Fermi profile $\rho_{TF}$ arises  from solving the elliptic PDE,
\begin{equation*}
0 = \Delta \rho_{TF} + {1\over \e^2} ( \mathcal{K}^2_\e(x) - \rho_{TF}^2) \rho_{TF}
\end{equation*}
on $\R^2$. Jerrard-Smets provided a rigorous proof of \eqref{eq:ODEbec} in \cite{JSm}.

\subsection{Critical scaling dynamics}
Dipoles in BECs have been created in physical experiments, by using the Kibble-Zurek mechanism  \cite{freilich2010real} and by dragging non radially symmetric BEC's through laser obstacles in \cite{neely2010observation}, and in numerical experiments \cite{MKFCS}.  
  These studies show that dipoles interact in nontrivial ways with each other and the background potentials.  Reduced ODE models for the dynamics of dipole configurations in BECs were proposed in \cite{MTKFCSFH, SKS, SMKS, torres2011dynamics,Torres} and  agree with numerical simulations of the full equation \eqref{eq:igl} and with physical experiments, \cite{MTKFCSFH}.  These ODE models, described below, were rigorously proven in \cite{KMS}.

There is a critical asymptotic regime where vortices interact with both the background potential $ \mathcal{K} _\varepsilon(x)$ and each other.   
This corresponds to studying how either vortices transition over small material defects, whose size is related to the length scale associated to the vortex cores, or how BEC vortices behave in anisotropic potentials, which have been physically realized, see \cite{freilich2010real, neely2010observation}.

Along the lines of the reduced ODE's \eqref{eq:ODEsf} and \eqref{eq:ODEbec},  a set of simple ODE's with anisotropic traps were proposed and  studied in which vortices interact with the background potential and with each other, see \cite{SMKS}.  Note also the earlier work on vortices in BEC in isotropic potentials, \cite{SKS, torres2011dynamics, Torres}.  A rigorous proof was proved by \cite{KMS} in a specific critical regime of inhomogeneities. When $\mathcal{K}_\e(x)$ is  asymptotically close to 1, described below, then vortices interact with both the background and each other, and outside of the range one finds induce dynamics that are dominated by only the background potential or by only vortex-vortex interactions. In particular, if one sets $\mathcal{K}^2_\e(x) = 1 + {Q_\e \over \logep}$ and if $Q_\e(x) \to Q_0(x)$ in a sufficiently smooth topology, then one finds that vortices $a_j$ move according to the Hamiltonian system,  
\begin{equation} \label{eq:ODE}
\pi d_j \dot{a}_j = \nabla^\perp_{a_j} H(a, d),
\end{equation}
where 
\begin{equation}\label{H}
H(a,d) = W(a,d) + \pi  \sum^{n}_{j=1} Q_0(a_j) 
\end{equation}
and $W(a,d)$ is the Coulomb potential defined in \eqref{eq:coulomb}. Here $Q_0(x)$ is the limiting rescaled background potential.   
The following result of \cite{KMS} establishes links between \eqref{eq:igl} and \eqref{eq:ODE}-\eqref{H}:

\begin{theorem}[\cite{KMS}]
Let $ \mathcal{K}_\e^2 = 1 +{Q_\e(x) \over \logep}$ where the $Q_\e \to Q_0$ in $H^4$.
Let $\{a_{j}^0, d_j\}$ be a configuration of vortices such that $d_j \in \{-1, 1\}$ and suppose $u_\e^0, 0 < \e <1$ is 
well-prepared initial data.
If $u_\e(t)$ is a solution to \eqref{eq:igl} with initial data $u_\e^0$,
then there exists a time $T > 0$    
such that for all $t \in [0,T]$, $u_\e(x,t)$ is asymptotically close to \eqref{eq:vortexform} with vortex locations given 
by a solution to the ODE
\eqref{eq:ODE}-\eqref{H}\footnote{Although the proof in \cite{KMS} is for bounded domains, the argument can be adapted to $\Omega \equiv \R^2$  by integrating methods for Gross-Pitaevsky vortex dynamics on $\R^2$.  See Remark~1.1 of \cite{KMS}.}.
\end{theorem}

The notion of well-preparedness requires that the initial data have the asymptotically-correct amount of energy to generate the $n$ vortices.

When there are only two vortices present on $\R^2$ then it is straightforward to see that  $T = \infty$; hence, dipoles will never collide in finite time, so long as $Q_0$ is sufficiently smooth.   
We note that vortex dipoles that interact with each other and the background potential have been created experimentally in \cite{neely2010observation}.

\subsection{Setup of the dipole problem}
In this paper, we consider dipoles in a critical regime in which vortices not only interact with each other but also with the background potential. The inverse problem we are interested in is to reconstruct inhomogeneous background potential $Q_0(x)$ from the trajectories of the dipole. % which are characterized by (\ref{e:comeqn1}) and (\ref{e:comeqn2}). 
We note that the detailed setting of this inverse problem is in section \ref{IP}. The setup of the dipole problem and its associated ODE system are discussed as follows.  

Consider a pair of vortices with centers $a=\{a_+, a_-\}$ of opposite charge (vortex dipole) $d=\{d_+, d_-\}=\{1, -1\}$. We replace $a_1$ and $a_2$ in the limiting ODE (\ref{eq:ODE}) by $a_+$ and $a_-$, respectively, and obtain the following equations:
\begin{align}\label{di2}
\pi d_\pm \dot{a}_{\pm} = \nabla^\perp_{a_\pm} H(a,d), 
\end{align}
where the function $H$ defined in (\ref{H}) is
\begin{align*}
H(a,d) = W(a,d) + \pi ( Q_0(a_+)+Q_0(a_-))  
\end{align*}
with the Coulomb potential  
$$
W(a, d) = - d_+ d_- \log |a_+ - a_-| = \log |a_+ - a_-|
$$ 
on $\R^2$.

A direct computation of (\ref{di2}) gives the following evolution:
\begin{align}\label{dipolepair}
 \dot{a}_+ & = {1\over \pi} { (a_+ - a_-)^\perp \over |a_+ - a_-|^2 } + \nabla^\perp Q_0 (a_+),  \\\label{dipolepair1}
-  \dot{a}_- & = -{1\over \pi}  { (a_+ - a_-)^\perp \over |a_+ - a_-|^2 } + \nabla^\perp Q_0(a_-) .  
\end{align}

To simplify the problem, we will track the centers of mass $q$ and the travel direction $p^\perp$ which are evaluated as follows:
\begin{align*}
p(t) & = {1\over 2} \LC a_+(t) - a_-(t)\RC \\
q(t) & = {1\over 2} \LC a_+(t) + a_-(t) \RC.
\end{align*}
Here, $2p$ is the dipole displacement.  From (\ref{dipolepair}) and (\ref{dipolepair1}), we obtain a new ODE system for $p$ and $q$:  
\begin{align}
\label{e:comeqn1}
2 \dot{p} & = \nabla^\perp Q_0 (q+p)  + \nabla^\perp Q_0 (q-p)\\
\label{e:comeqn2}
2 \dot{q} & = {1\over \pi} { p^\perp \over |p|^2} + ( \nabla^\perp Q_0 (q+p)  - \nabla^\perp Q_0 (q-p)).  
\end{align}
The corresponding Hamiltonian is 
\begin{equation}\label{H1}
H(q, p^\perp) = H(q_1,q_2; p_2,-p_1) \equiv {1\over 2 \pi} \log |p| +{1\over 2} \LC Q_0(q+p) + Q_0 (q - p) \RC.
\end{equation}

Note that when the background potential $Q_0$ is constant, the dipole displacement $2p$ is a constant vector, which from \eqref{e:comeqn2} implies that a dipole will travel with constant velocity and fixed direction $p^\perp$. On the other hand, if the background potential is variable, the inhomogeneities will refract the dipole paths. The behavior of the ODE system generated by dipoles in an anisotropic potential is nontrivial, see figure~\ref{fig1}.
%\begin{comment}
\begin{figure}[htbp] 
   \centering
   \includegraphics[width=4in]{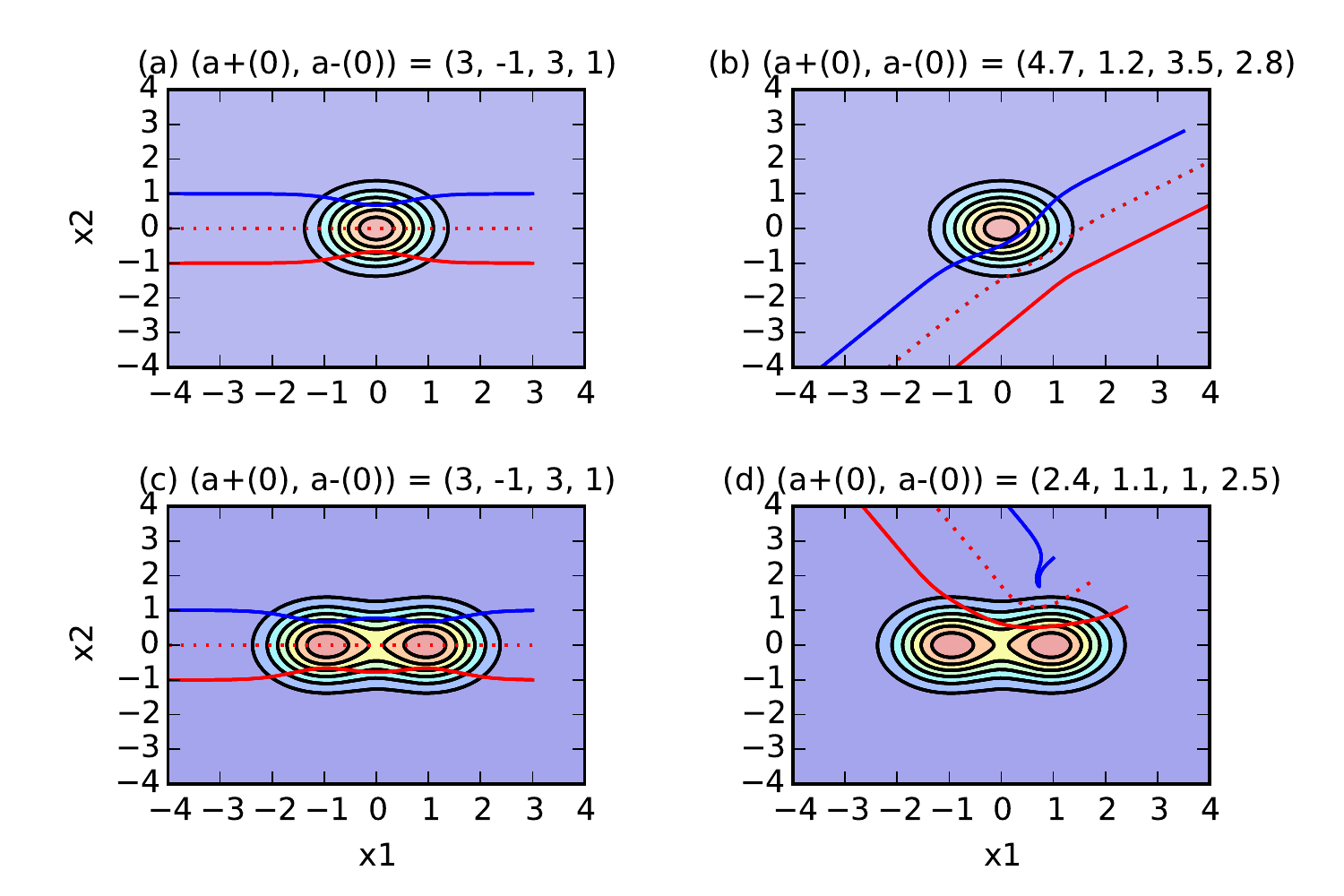} 
   \caption{Trajectories of the dipoles (solid lines) and of their centers of mass (dotted line) off anisotropic background potentials $0.1e^{-|(x_1, x_2)|^2}$ and $0.1e^{-|(x_1, x_2)-(1,0) |^2}+0.1e^{-|(x_1, x_2) +(1,0) |^2}$.  }
   \label{fig1}
\end{figure}
%\end{comment}
The  dynamical behavior of dipoles in the presence of anisotropic backgrounds is extremely rich and has been studied, see for example \cite{goodman2015dynamics}.

\section{An Inverse Problem: The Main Result}\label{IP}	
 
Let $\Omega$ be an open and bounded domain in $\R^2$. Let $B_{\rho/2\pi}\supset\ov\Omega$ be an open ball with center at the origin and the radius $\rho/2\pi>0$. We denote the  boundary of $B_{\rho/2\pi}$ by $\pd B_{\rho/2\pi}$.  To recover a background potential $Q_0$ which is constant outside $\Omega$, we will first send vortices with initial positions outside $B_{\rho/2\pi}$ into $\Omega$.  Then, after the vortices exit $\Omega$ and $B_{\rho/2\pi}$, we will use their positions and travel directions on $B_{\rho/2\pi}$ to deduce $Q_0$. % Our uniqueness result will be that two potentials $Q^1_0$ and $Q^2_0$ recovered in this way are identical if they are known to be sufficiently flat.

We denote the center of mass by $x(t)$ and the travel direction by $\xi(t)$; that is, $x(t) = q(t)$ and $\xi(t) = p(t)^\perp$.  From (\ref{e:comeqn1})-(\ref{H1}), the trajectories of the center of mass and the travel direction are modeled by the following Hamiltonian system  
\begin{align}\label{odesys}
        & \dot{x} (s) = {\partial H \over \partial \xi},  \ \ \
         \dot{\xi} (s) = -{\partial H \over \partial x},
\end{align}
where the Hamiltonian is 
\begin{align}\label{Hamil}
   H(q,p^\perp) \equiv H(x,\xi) = {1 \over 2 \pi} \log |\xi| + {1 \over 2} ( Q_0(x + \xi^\perp) + Q_0 (x - \xi^\perp) ).
\end{align}

For initial conditions $(x(0),\xi(0))=(x_0,\xi_0)=:X^{(0)}$, we let direction $\xi_0\in \mathbb{S}^1=\{\theta\in\mathbb{R}^2: |\theta|=1\}$ and position $x_0$ satisfy $x_0\cdot\xi_0=-\rho/2\pi$.  Equivalently, for each $\xi_0\in\mathbb S^1$, we let $x_0=z-\frac{\rho}{2\pi}\frac{\xi_0}{|\xi_0|}$, where $z\in\re^2$ satisfies $z\cdot\xi_0=0$.  When $z=0$, $x_0$ is on $\pd B_{\rho/2\pi}$, opposite to direction $\xi_0$.  For large $z$, the vortex trajectories will not intersect $B_{\rho/2\pi}$.  See Figure \ref{f:dipole}.

Suppose that $H_j$ are the Hamiltonians corresponding to two different limiting background potentials $Q_0^j$ for $j=1,2$, respectively.  We denote $(x_j, \xi_j)$ to be the solutions of the following Hamiltonian system with the same initial conditions $X^{(0)}$:
\begin{align*} 
    \left\{
         \begin{array}{ll}
         \dot{x}_j (s) = {\partial H_j \over \partial \xi} & \\
         x_j (0) = x_0  & \\
         \end{array}
    \right.
\ \ \ \hbox{and}\ \ \ \ \
    \left\{
     \begin{array}{ll}
     \dot{\xi}_j (s)=-{\partial H_j \over \partial x} & \\
     \xi_j (0)=\xi_0. & \\
     \end{array}\right.
\end{align*}
More precisely, $(x_1, \xi_1) $ and $(x_2, \xi_2)$ are the solutions to the systems :
\begin{align} \label{e:xeqn}
    \left\{
     \begin{array}{ll}
     \dot{x}  (s)= {1\over 2\pi} { \xi  \over |\xi|^2}  -{1\over 2} ( \nabla^\perp Q _0 (x+\xi^\perp)  - \nabla^\perp Q _0 (x-\xi^\perp)  )   \\
     x  (0)=z -\frac{\rho}{2\pi}\frac{\xi_0}{|\xi_0|}   \\
     \end{array}\right.
\end{align}
and
\begin{align} \label{e:xieqn}
    \left\{
         \begin{array}{ll}
         \dot{\xi} (s) = -{1\over 2}( \nabla Q_0 (x + \xi^\perp)  +  \nabla Q _0 (x - \xi^\perp) ) \\
         \xi(0) = \xi_0  \\
         \end{array}
     \right.
\end{align}
with  $Q_0$ replaced by $Q^1_0$ and $Q^2_0$, respectively. The solutions of \eqref{e:xeqn}-\eqref{e:xieqn} are denoted by 
$$
    X_j(s, X^{(0)}) = X_j(s, x_0, \xi_0) = (x_j(s) , \xi_j(s)),\ \ j=1,2
$$
with initial condition $X^{(0)}$ and limiting background potentials $Q_0^j$.

For each $X^{(0)}$ described above and each $j=1,2$, let
\begin{align}\label{time}
    t^j \equiv t^j(X^{(0)}) 
    &= \sup\Big\{s> 0:\{x_j(s)\pm \xi_j^\bot(s)\}_{\pm=+,-}\cap \ov B_{\rho/2\pi}\neq\varnothing\Big\}
\end{align}
be the time when both vortices $x_j(s)\pm\xi_j^\bot(s)$ have exited $B_{\rho/2\pi}$.  If no such $s$ is in the above set (i.e., large $z$), then $t^j=0$.  If $Q_0$ is constant outside $\Om$ and sufficiently flat, then $t^j$ is finite for each such $X^{(0)}$, and by compactness, the largest such time
\eqal{\label{tau}
\tau^j=\sup_{\substack{|\xi_0|=1,\\z\cdot\xi_0=0}}t^j(X^{(0)})
}
is finite.

Our main result, which is stated in the following theorem, is the uniqueness of the potential $Q_0$ from the dipole's trajectory.  The proof is given in Section \ref{nonlinear}. 
 
\begin{theorem} \label{q1}
Suppose that $Q_0^j\in C^8(\Omega)$ for $j = 1,2$ are background potentials that satisfy $Q_0^1 = Q_0^2 = \hbox{constant}$ in $\R^2 \backslash \Omega$.
Suppose that $t^1 = t^2$ and $X_1(t^1, X^{(0)}) = X_2(t^2, X^{(0)})$ for all $X^{(0)}$ given above.
Then there is a sufficiently small $\varepsilon > 0$ such that if
\begin{align}\label{appro3}
      \nabla  Q_0^1-\nabla  Q_0^2 \in{C^{7}_0( \Omega )},\ \ \ \| \nabla  Q_0^j \|_{C^{7}( \Omega)}\leq \varepsilon,
\end{align}
then $Q^1_0 \equiv Q^2_0$. (see Figure~\ref{f:dipole})
\end{theorem}

%In addition, we consider the problem of explicitly reconstructing the background potential $Q_0$ from scattering data $X(\tau,X_0)$.  Our approach, detailed in section \ref{sec:recon}, is to linearize the inverse problem provided that $\|\nabla Q_0\|$ is small.  The resulting linear integral equations are then solved analytically, yielding an approximate reconstruction formula \eqref{recon0}.  The result is presented in Theorem \ref{recon_Q}, see section \ref{sec:recon}.

\begin{figure}[htbp] %  figure placement: here, top, bottom, or page
   \centering
   \includegraphics[width=2.5in]{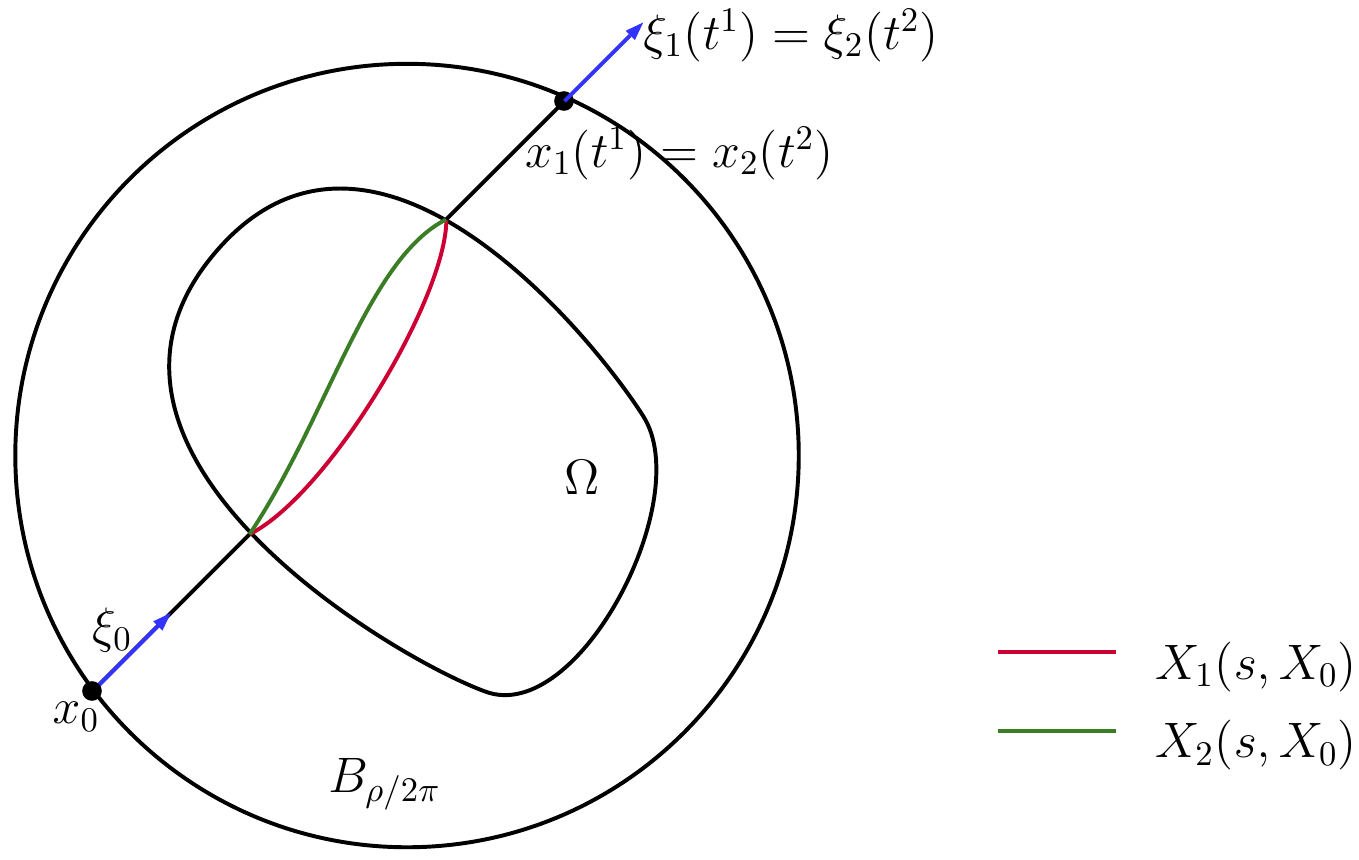} 
   \caption{Dipole Refraction off  inhomogeneities: The image of the trajectories of the centers of mass $X_j=(x_j,\xi_j)$ $(j=1,2)$ with the initial position $x_0$ and direction $\xi_0$.}
   \label{f:dipole}
\end{figure}

The inverse problem stated in Theorem~\ref{q1} is a way to generate an effective potential by the examination of the trajectories of sufficient numbers of dipoles.  An interesting example of where this  could be 
useful is in BEC's with random impurities generated by optical speckle potentials or quasi-periodic optical lattices \cite{FallaniLBcia}.  Their associated  trapping potentials with nontrivial characteristics  are not easily deduced \cite{ KondovSS2011TAlo, ModugnoGiovanni2010AliB, RoatiGiacomoAloa}, but such potentials could be inferred using the methods described, and they may also yield  measures of disorder in these impure BEC's \cite{PhysRevA.93.013607}.

\subsection{Methodology and relationship to the rigidity problem}\label{sec:mth}

The inverse problem we address here is to determine the background potential by knowing the first arrival time, the exit point and direction of the dipole's path if one knows its point and direction of entrance.

This inverse problem is closely related to the lens rigidity problem which consists of determining the Riemannian metric from the scattering relation
which measures, besides the travel times, the point and direction of exit of a geodesic from the manifold.
For simple metrics, it was observed by Michel \cite{Michel} that the lens rigidity problem is equivalent to the boundary rigidity one which consists of determining the metric by knowing the boundary distance function between boundary points. This kind of problem arose in geophysics in attempting to study the inner structure of the Earth by measuring the travel time of seismic waves going through the Earth. We refer to \cite{SUV} and the references therein for the recent development of the rigidity problem.

Our investigation on the reconstruction of the background potential from the dipole's trajectories is inspired by the results of Stefanov and Uhlmann \cite{SU3} on the boundary rigidity problem. They showed that the Riemannian metrics can be determined from the lengths of the geodesics if the metrics are close to the Euclidean one. An important step to recover  the metric in \cite{SU3} is the establishment of an identity which connects the difference of two Riemannian metrics with the scattering relations. 
A similar identity was used later to prove the stability estimate of the rigidity problem in \cite{W}. Moreover, a numerical algorithm was designed in \cite{CQUZ} for the isotropic index of refraction from travel time measurements. This identity works not only for geodesics which was considered in the works \cite{CQUZ, CQUZ2011, PSUZ, SU3, SUV, UW03, W}, but also for any kind of curves under suitable assumptions. 

This paper is mainly devoted to analyzing an anisotropic background potential using the trajectories of vortex dipoles. 
We first derive from (\ref{dipolepair}) and (\ref{dipolepair1}), the ODE system (\ref{e:xeqn}) and (\ref{e:xieqn}) with the Hamiltonian (\ref{Hamil}) which describes the travel trajectories of the centers of mass of the dipole.  
The reduction to an ODE system combined with a suitable assumption on exit information around the boundary of the domain leads to an identity which reveals that the variation in potentials is related to their scattering data. The main result we prove in this paper is the uniqueness of the potential. % and its reconstruction.
%enables us to apply the Stefanov-Uhlmann identity in our dipole problem by assuming the exit information can be measured around the boundary of the domain.
To make our approach clear, we start from considering a linearization of the problem and then discuss the setting for the potential in an inhomogeneous background. %In addition, we also present numerical reconstruction for different potentials.
The idea in the linearized setting is to derive from \eqref{F2int1} that the integral along the trajectory of the difference of the two potentials is zero. Then we recover the potential by applying the standard Fourier transform. 
However, it is not an obvious path to extend the same method to the study of the potential in an inhomogeneous background. To address the anisotropic potential, we use a perturbation argument and the identity (\ref{F2int1}) to derive $L^2$ boundedness of a Fourier integral operator. Therefore, the uniqueness of the potential comes from the inversion of this operator.
 
This paper is organized as follows. In section \ref{sec:SU}, we study the key identity which links two background potentials with the travel information. The uniqueness of the potential for the linearized setting of the inverse problem is presented in section \ref{sec:linear}. We extend the idea from linearized case to the anisotropic setting and prove Theorem \ref{q1} in section \ref{nonlinear}. A reconstruction formula is given in section \ref{sec:recon}, with numerical examples in section \ref{sec:numer}.

% % % % % % % % % % % % % % % % % % % % % % % % % % % % % % % % % % % % % % % % % % % % % % % % % % % % % %

\section{The key Stefanov-Uhlmann identity}\label{sec:SU}
In this section, we introduce an important identity from Stefanov and Uhlmann's paper \cite{SU3}. The identity derived in \cite{SU3} for the boundary rigidity problem relies on the hypotheses that two Riemannian metrics are close to the Euclidean metric and have the same boundary distance function. Since the boundary rigidity problem and the lens rigidity one are equivalent for simple metrics, if we assume the exit information and travel time are known, then the same identity can be derived in the same manner as in \cite{SU3}. We state this identity in the following lemma and include its proof for completeness. With this identity, we are able to derive the uniqueness of the background potential.

\begin{lemma} \label{lemmaid} 
Under the hypothesis in Theorem \ref{q1}, we have
 \begin{align} \label{F2int1}
       \int^t_0 {\partial X_2 \over \partial X^{(0)}} (t-s, X_1(s, X^{(0)}))(V_1 - V_2)(X_1(s,X^{(0)} ))ds = 0,
 \end{align}
where  
$$
 V_j := (\partial_\xi H_j, -\partial_x H_j)^T,\ \ j = 1,2 
$$
are the Hamiltonian vector fields. Here $t\equiv t(X^{(0)})=\max\{t^1,\ t^2\}$ with $t^1,t^2$ defined in (\ref{time}) for the initial starting data $X^{(0)}$ and the superscript $T$ stands for the transpose of a vector or matrix. 
\end{lemma}

\begin{proof}
Let $F(s) = X_2 (t-s, X_1(s, X^{(0)}) )$. Here $t = t(X^{(0)})$ is independent of the potential $Q_0^1$ or $Q_0^2$.
From the hypothesis of Theorem \ref{q1}, the two trajectories of the centers of mass $X_j,\ j=1,2$ with the same initial $X^{(0)}$ have the same point and direction of exit and travel time. This implies $F(0) = X_2 (t, X^{(0)}) = X_1(t,X^{(0)}) = F(t)$. Hence, we have
\begin{align*}
    \int^t_0 F'(s) ds = 0.
\end{align*}
Moreover, the function $F'$ is computed as follows:
\begin{align*} 
    F'(s)
    & = -V_2 (X_2(t-s, X_1(s, X^{(0)}) ) + {\partial X_2 \over \partial X^{(0)}} (t-s, X_1(s,X^{(0)}) ) V_1 ( X_1(s,X^{(0)}) )\\
    & = {\partial X_2 \over \partial X^{(0)} } (t-s, X_1(s, X^{(0)}) ) (V_1 - V_2)( X_1 (s,X^{(0)}) ),
\end{align*}
where the second identity is due to the fact that
$$
       0 = \left. {d \over ds} \right|_{s = 0} X (T-s, X(s, X^{(0)})) = - V(X(T, X^{(0)})) + {\partial X \over \partial X^{(0)}}(T, X^{(0)}) V (X^{(0)})
$$
for all $T$. Therefore, the identity (\ref{F2int1}) holds.
\end{proof}

We make an adjustment to \eqref{F2int1} by replacing $t(X^{(0)})$ with the large, constant time $\tau^1=\tau^2=\tau$, see \eqref{tau}, since this parameter will commute with integrals.  We claim that $X_1(\tau,X^{(0)})=X_2(\tau,X^{(0)})$ in the above situation. Since $t(X^{(0)})=\max\{t^1,\ t^2\}$ is the final time the vortices interact with the background potential, $Q_0^j(x_j(s)\pm\xi_j^\perp(s))\equiv constant$ for all $s\ge t(X^{(0)})$, so the later positions at $s=\tau$ arise from those at $s=t(X^{(0)})$ via straight lines.  Since solutions are unique, the trajectories are identical after $t(X^{(0)})$.  Therefore,
\begin{align} \label{id:tau}
       \int^\tau_0 {\partial X_2 \over \partial X^{(0)}} (\tau-s, X_1(s, X^{(0)}))(V_1 - V_2)(X_1(s,X^{(0)} ))ds = 0 
 \end{align}
for all given $X^{(0)}$.

\section{The Linearized Case}\label{sec:linear}
In this section, the linearized version of this dipole inverse problem is studied. We will apply the identity (\ref{id:tau}) to derive the uniqueness of the background potential in this setting.

Let $\rho > 0$ be chosen so that $\overline{\Omega} \subset B_{ \rho/ 2 \pi}$, with $|\xi_0|=1$ and $z\cdot\xi_0=0$.  Let $g(x) = {1\over 2} \nabla  Q_0(x)$ and $g^\perp(x) = {1\over 2} \nabla^\perp  Q_0(x)$.  Similar to (\ref{e:xeqn}) and (\ref{e:xieqn}), we have the following Hamiltonian system:
 \begin{align}\label{e:xg}
     \left\{
         \begin{array}{ll}
        \dot{x}(s) = {1 \over 2 \pi} { \xi  \over |\xi|^2} - ( g^\perp(x + \xi^\perp) - g^\perp(x - \xi^\perp))&  \\
        x(0)= z - {\rho \over 2 \pi}\,\frac{\xi_0}{|\xi_0|} =: x_0& \\
        \end{array}
    \right.
 \end{align}
 and
\begin{align}\label{e:xig}
    \left\{
         \begin{array}{ll}
         \dot{\xi} (s) = -( g(x + \xi^\perp) + g(x - \xi^\perp))&\\
         \xi(0) = \xi_0.& \\
         \end{array}
     \right.
\end{align}

We first consider the dipole problem in the absence of the gradient of the background potential. Assume that $g(x) \equiv 0$. Then the solution to \eqref{e:xg}-\eqref{e:xig} is $X_e(s, X^{(0)}) = (x_e(s, X^{(0)}), \xi_e(s, X^{(0)}))$ with initial $X^{(0)} = (x_0, \xi_0)$. More precisely, the solution $X_e(s, X^{(0)})$ can be written as\footnote{Note that in \cite{SU3}, they have $X_e(s,X^{(0)}) = (x_0 + s \xi_0, \xi_0)$.}   
\begin{equation} \label{e:zeroth0}
X_e(s,X^{(0)}) = \left( x_0 + {s\over 2 \pi} { \xi_0 \over |\xi_0|^2} ,\  \xi_0 \right)
\end{equation}
for all $s$.
We formally differentiate $X_e(s,X^{(0)})$ with respect to $X^{(0)}$ and obtain a $4 \times 4$ matrix
\begin{equation*}
{\p X_e(s, X^{(0)}) \over \p X^{(0)}} 
= \begin{pmatrix}  I_2 &    s\,E(\xi_0)  \\  0 & I_2 \end{pmatrix},
\end{equation*} 
where
\eqal{\label{Emat}
E(\xi_0)=\frac{1}{2\pi}\left[{I_2 \over |\xi_0|^2} - { 2 \xi_0 \otimes  \xi_0 \over |\xi_0|^4}\right].
}

Following the argument in \cite{SU3}, we first do a formal linearization of the problem about the trivial state in \eqref{e:zeroth0}. For $\varepsilon > 0$,  the linearized expression is:
\begin{align*}
x^\e & =  x_0 + {s \over 2 \pi} { \xi_0 \over |\xi_0|^2} + \e \,\tilde x, \ \ \
\xi^\e = \xi_0 + \e \,\tilde \xi, \\
X^\e & = (x^\e, \xi^\e),\ \ \
g^\e  = \e \tilde g.
\end{align*}
The difference of the gradient of two potentials $Q^j_0$, $j=1,2$, is denoted by
\[
m^\e = g^\e_1 - g^\e_2 = \e \tilde m. 
\]
The Fr\'echet derivative on our identity \eqref{id:tau} yields
\begin{align*}
0 = \int_{0}^\tau \left. {d \over d \e} \right|_{\e = 0} \LB { \p X_2^\e \over \p X^{(0)}} ( \tau - s, X_1^\e(s, X^{(0)})) \cdot ( V_2^\e - V_1^\e)(X_1^\e(s, X^{(0)})) \RB ds.
\end{align*}

To compute this, first note that 
\[
\left.( V_2^\e - V_1^\e)(X_1^\e(s,X^{(0)}))  \right|_{\e = 0} = \begin{pmatrix} 0 \\ 0 \end{pmatrix}, 
\]
so we only need to take the variation with respect to the difference $V_2^\e - V_1^\e$.  In particular, we have 
\begin{align*}
  \tilde V(s, x_0, \xi_0)  := &\left. {d \over d \e} \right|_{\e = 0} \LB ( V_2^\e - V_1^\e)(X_1^\e(s, X^{(0)})) \RB\\
  = &\begin{pmatrix} 
 \tilde m^\perp (x_e + \xi_0^\perp) - \tilde m^\perp ( x_e - \xi_0^\perp) \\
 \tilde m (x_e + \xi_0^\perp) + \tilde m ( x_e - \xi_0^\perp)
 \end{pmatrix},
\end{align*}
where $x_e=x_e(s,X^{(0)})$ is as in \eqref{e:zeroth0}.  Therefore, the final linearized identity is
\begin{align}\label{3:mid}
\begin{split}
& \int_{0}^\tau \begin{pmatrix}      I_2 &  (\tau-s)E(\xi_0)  \\  0 & I_2 \end{pmatrix}
 \tilde V(s, x_0, \xi_0) ds = 0.
\end{split}
\end{align}

With the above linearized identity, we will be able to show uniqueness for the linearizing dipole problem. The result is stated as follows. 
\begin{theorem}
Suppose that $\supp \LC  \tilde m \RC \subset \Omega$. If \eqref{3:mid} holds for all $ \xi_0\in \mathbb{S}^1$ and $z\in\re\,\xi_0^\perp$,
then $\tilde m=0$ in $\Omega$.  
\end{theorem}
Note that $\re\xi_0^\perp=\{\lambda \xi^\perp_0:\lambda\in\re\}$.

\begin{proof}
Since the dipole pair will exit the ball $B_{\rho/2\pi}$ after time $\tau$ by the definition of $\tau$, and $\supp \LC \tilde m \RC \subset B_{\rho/ 2\pi}$, we can extend the integration region of (\ref{3:mid}) to $(-\infty, \infty)$. Hence, we have
\begin{align*} 
& \int_\re
\begin{pmatrix}      
I_2 & (\tau-s)E(\xi_0)  \\  
0 & I_2 \end{pmatrix}
\begin{pmatrix} 
\tilde m^\perp (x_e + \xi_0^\perp) - \tilde m^\perp ( x_e - \xi_0^\perp) \\
\tilde m (x_e + \xi_0^\perp) + \tilde m ( x_e - \xi_0^\perp) 
\end{pmatrix}ds = 0.
\end{align*}
Note that $x_e(s) = z - {\rho\over 2\pi} {\xi_0\over|\xi_0|} + {s \over 2\pi} {\xi_0 \over |\xi_0|^2}$. Let $\eta\in\re^2\setminus\{0\}$ satisfy $\eta \cdot \xi_0 = 0$; clearly, $\eta\cdot z=\eta\cdot x_e$. Taking the Fourier transform of the above integral with respect to $z$, we obtain
\begin{align*} 
& \int_{z\cdot\xi_0=0}\int_\re e^{-i\eta\cdot x_e(s,X^{(0)})}
\begin{pmatrix}      
I_2 & (\tau-s)E(\xi_0)  \\  
0 & I_2 
\end{pmatrix}
\begin{pmatrix} 
\tilde m^\perp_+ - \tilde m^\perp_- \\
\tilde m_+ + \tilde m_- 
\end{pmatrix}ds\,dS_z = 0,
\end{align*}
where $\tilde m_\pm=\tilde m(x_e\pm\xi_0^\perp)$.  Choose $\xi_0^\bot\equiv\xi_\eta^\bot=\eta/|\eta|$, where the subscript emphasizes dependence on $\eta$. Let $(s,z)\to x_e$ via
%\RS{Added $|\xi_\eta|$ to $s$}
\eqal{\label{sz}
s&= |\xi_\eta| \,\rho+2\pi\xi_\eta\cdot x_e,\\
z&=x_e-\frac{\xi_\eta\cdot x_e}{|\xi_\eta|^2}\,\xi_\eta,
}
we get
\begin{align*}
\int e^{-i\eta\cdot x_e}(\tilde m_+-\tilde m_-)dx_e&=E^\bot(\xi_\eta)\int e^{-i\eta\cdot x_e}(\tau-s(x_e))(\tilde m_++\tilde m_-)dx_e,\\
\int e^{-i\eta\cdot x_e}(\tilde m_++\tilde m_-)dx_e&=0,
\end{align*}
where $E^\bot=\begin{pmatrix}0&1\\-1&0\end{pmatrix}E$ satisfies $E^\bot m=(Em)^\bot$.  Adding and subtracting these equations:
\begin{align}\label{comp}
\int e^{-i\eta\cdot x_e}\tilde m(x_e\pm\xi_\eta^\bot)dx_e&=\pm\frac{1}{2}\,E^\bot(\xi_\eta)\int e^{-i\eta\cdot x_e}(\tau-s(x_e))[\tilde m(x_e+\xi^\bot_\eta)+\tilde m(x_e-\xi^\bot_\eta)]dx_e.
\end{align}
We change variables $y_\pm=x_e \pm\xi_\eta^\bot$ in the appropriate integrals such that $\tilde m_\pm=\tilde m(y_\pm)$.  Noting that $s(x_e)=s(y_\pm)$ and $e^{-i\eta\cdot x_e}=e^{\pm i|\eta|-i\eta\cdot y_\pm}$, we get
\[
e^{\pm i|\eta|}\int e^{-i\eta\cdot y_\pm}\tilde m(y_\pm)dy_\pm=\pm{E}^\bot(\xi_\eta)\cos(|\eta|)\int e^{-i\eta\cdot y}(\tau-s(y))\tilde m(y)dy,\qquad \pm=+,-,
\]
which implies
\eqal{\label{ftL}
e^{\pm i|\eta|}\widehat{\tilde m}(\eta)&=\pm E^\bot(\xi_\eta)\,[(\tau-\rho)\cos(|\eta|)\,\widehat{\tilde m}(\eta)-2\pi i\,\cos(|\eta|)\,(\xi_\eta\cdot\nabla_\eta)\widehat{\tilde m}(\eta)],
}
where we use that $\tau$ is a constant and denote the Fourier transform of $f$ by $\widehat{f}$.  Since
\eqal{\label{cos0}
\cos(|\eta|)\,\widehat{\tilde m}(\eta)=0, 
}
the first term in \eqref{ftL} vanishes.  Moreover, if we apply $\xi_\eta\cdot \nabla_\eta$ to \eqref{cos0}, then we obtain
$$
0=(\xi_\eta\cdot\nabla_\eta)(\cos(|\eta|)\,\wh{\tilde m}(\eta))=\cos(|\eta|)\,(\xi_\eta\cdot\nabla_\eta)\wh{\tilde m}(\eta),
$$
which implies the second term in \eqref{ftL} also vanishes. Thus, $\widehat{\tilde m}(\eta)=0$ on $\re^2\setminus\{0\}$.  Since $\wh{\tilde m}(\eta)=\frac{1}{2\eps}\,i\eta\,(\wh Q^1_0(\eta)-\wh Q^2_0(\eta))$, we conclude $\wh{\tilde m}(\eta)=0$ on $\re^2$.  Hence, $\tilde m\equiv 0$.

\end{proof}

% % % % % % % % % % % % % % % % % % % % % % % % % % % % % % % % % % % % % % % % % % % % % % % % % % % % % % % % % % % % % % % % % % % % % % % % % % % % % % % % % % % % % % % % % % % % % % % % % % % % % % % % % %

\section{The reconstruction of the anisotropic background potential}\label{nonlinear}
 
In this section, we investigate the dipole inverse problem in an inhomogeneous background. In this setting, the background potential is not a constant; therefore, from the ODE system (\ref{e:xg}) and (\ref{e:xig}), we observe that the dipole's trajectories are affected by potentials; consequently, dipoles do not, in general, travel in straight lines. 

\subsection{The weakly nonlinear inverse problem}
Suppose that we have two background potentials $Q_0^j$ for $j=1, 2$. The gradients $g_j=\frac{1}{2}\nabla Q_0^j$ satisfy
$$
     g_1(z)-g_2(z)\in C^{k-1}_0(\Omega)\ \ \hbox{and}\ \ \|g_j\|_{C^{k-1}(\Omega)}\leq \varepsilon
$$ 
for $j=1,2$.
We extend $g_1$ and $g_2$ to $\mathbb{R}^2$ such that $g_1=g_2=0$ in $\mathbb{R}^2\setminus B_{\rho /2\pi}$,  and $g_j\in C^{k-1}(\mathbb{R}^2)$ for $j=1,2$. The extended functions $g_j$ satisfy
\begin{align}\label{appro1}
      g_1-g_2\in{C^{k-1}_0( \mathbb{R}^2 )} \ \ \hbox{and}\ \  \|g_j\|_{C^{k-1}( \mathbb{R}^2)}\leq \varepsilon.
\end{align}
From now on, we use the same symbols $g_1$ and $g_2$ to denote the extended gradient of the background potentials $Q_0^j$ which satisfy the above conditions.

Let $X_j=(x_j, \xi_j)$, for $j=1,2$, be the solutions to (\ref{e:xg}) and (\ref{e:xig}) with $g$ replaced by $g_j$, respectively. Recall that if $g \equiv 0$, then from (\ref{e:zeroth0}) the solution is $X_e(s, X^{(0)}) = (x_0 + {s\over 2\pi}{\xi_0 \over |\xi_0|^2},\ \xi_0 )$ with $x_0=z-\frac{\rho}{2\pi}\frac{\xi_0}{|\xi_0|}$ and $z\cdot \xi_0=0$. For general functions $g_j \neq 0$ that satisfy (\ref{appro1}),
the solutions $X_1$ and $X_2$ can be expressed as small perturbations of the solution $X_e$: 
\begin{align}\label{est:x}
    x_j(s,X^{(0)}) = x_0 + {s\over 2\pi}{\xi_0 \over |\xi_0|^2}+O(\varepsilon)\ \ \hbox{in $C^{k-1}$}
\end{align}
and 
\begin{align}\label{est:xi}
   \xi_j(s,X^{(0)}) = \xi_0+O(\varepsilon) \ \ \hbox{in $C^{k-1}$,}
\end{align}
where the $C^j$ norm is interpreted with respect to variables $s,\ z,\ \xi_0 $ and $O(\varepsilon)$ is a function with norm bounded by $C\varepsilon$ with a constant $C > 0$ uniformly in any fixed compact set. 

Taking the derivatives of (\ref{est:x}) and (\ref{est:xi}) with respect to variables $X^{(0)} = (x_0,\xi_0)$, it follows that 
$$
     \LN {\partial X_{2}  \over \partial X^{(0)} } - 
     \left(\begin{array}{cc}
         I_2 &  s\,E(\xi_0)   \\
         0 &   I_2   \\
        \end{array}\right)
       \RN_{C^{k-2}}=O(\varepsilon),
$$
where $E$ is as in \eqref{Emat}.
We then obtain
\begin{align}\label{Bs} 
       {\partial X_{2} \over \partial X^{(0)} }(\tau-s, X_{1}(s, X^{(0)})) 
      = \left(\begin{array}{cc}
               I_2 &  (\tau-s)\,E(\xi_0)   \\
               0 &   I_2   \\
              \end{array}\right)
       +B(s)
\end{align} in \eqref{id:tau} with the $4\times 4$ matrix
 \begin{align} \label{Bs1}
     B(s)=B(s,X^{(0)}; Q_0^1,Q_0^2)=
     \left(\begin{array}{cc}
                       B_{11} &   B_{12} \\
                       B_{21} &   B_{22}   \\
     \end{array}\right)
     =O(\varepsilon)\ \ \ \hbox{in $C^{k-2}$}.
 \end{align}    
Here, $\tau$ in \eqref{tau} is the largest possible time $s$ for the vortices to exit $B_{\rho/2\pi}$.

We denote $\m=g_1-g_2$. From the assumptions (\ref{appro1}) on $g_j$, we get $\m \in C^{k-1}_0(\R^2)$ and $\supp(\m) \subset \Omega$. The difference of the Hamiltonian vector fields is
\begin{align}\label{V12}
            &(V_2 -V_1)(X_{1}(s,X^{(0)})) \notag\\ 
            &=\left( \begin{array}{c}
                \m^\perp(x_1+\xi_1^\perp)-\m^\perp(x_1-\xi_1^\perp)\\
                \m(x_1+\xi_1^\perp)+\m(x_1-\xi_1^\perp)\\
            \end{array}\right)  
            =: \left( \begin{array}{c}
             \textbf{M}_1(x_1,\xi_1^\perp)\\  
             \textbf{M}_2(x_1,\xi_1^\perp) \\
             \end{array}\right).
\end{align} 
Recalling \eqref{id:tau}, we can extend the integration region to the whole real line, since $\supp (\m)\subset\Om$. Therefore, we have 
 \begin{align}\label{intrho}
       \int_\re {\partial X_{2}\over \partial X^{(0)}}(\tau-s, X_{1}(s, X^{(0)}))\left( \begin{array}{c}
                    \textbf{M}_1(x_1,\xi_1^\perp)\\  
                    \textbf{M}_2(x_1,\xi_1^\perp) \\
                    \end{array}\right)ds=0,
 \end{align}
where $\partial X_2/\partial X^{(0)}$ is evaluated to $O(\varepsilon)$ in \eqref{Bs}.  Recall $x_e=z-{\rho\over 2\pi} {\xi_0\over |\xi_0|}+\frac{s }{2\pi}\frac{\xi_0}{|\xi_0|^2}$, where $z\cdot \xi_0=0$.  Let $\eta\in\re^2$ satisfy $\eta\cdot\xi_0=0$ and taking the Fourier transform in $z$ give
\begin{align*}
\int_{z\cdot\xi_0=0}\int_\R e^{-i\eta\cdot x_e}
\left(
\begin{pmatrix}
I_2& (\tau-s)E(\xi_0)\\
0&I_2
\end{pmatrix}+
\begin{pmatrix}
B_{11}& B_{12}\\
B_{21}&B_{22}
\end{pmatrix}\right)
\begin{pmatrix}
\M_1(x_1,\xi^\bot_1)\\
\M_2(x_1,\xi^\bot_1)
\end{pmatrix}
d s\,d S_z=0.
\end{align*}
Let $(s,z)\to x_e$ be as in \eqref{sz}. We obtain a system of equations:
\eqal{\label{sys}
\int_{\re^2} e^{-i\eta\cdot x_e}(\m(x_1+\xi_1^\bot)-\m(x_1-\xi_1^\bot))d x_e&=-(W^1\m+W^3\m)^\bot,\\
\int_{\R^2}e^{-i\eta\cdot x_e}(\m(x_1+\xi_1^\bot)+\m(x_1-\xi_1^\bot))d x_e&=W^2\m,
}
where $M_i=M_i(x_1,\xi_1^\bot)$ for $x_1 = x_1(x_e, \xi_0)$ and $\xi_1 = \xi_1(x_e, \xi_0)$, and the (small) terms on the right hand side are given by
\eqal{\label{Dis}
W^j\m(\eta)&=-\int_{\R^2}e^{-i\eta\cdot x_e}(B_{jk}\M_k)d x_e,j=1,2,\\
W^3\m(\eta)&=-E(\xi_0)\int_{\re^2}e^{-i\eta\cdot x_e}(\tau-\rho-2\pi\xi_0\cdot x_e)\M_2d x_e.
}
Adding and subtracting the equations in \eqref{sys} yields
\eqal{\label{eq}
\int e^{-i\eta\cdot x_e}\m(x_1\pm \xi_1^\bot)dx_e=\frac{1}{2}[W^2\m\mp(W^1\m+W^3\m)^\bot],
}
which contains small terms compared to \eqref{comp}. Our objective is to show that \eqref{eq} implies $\m=0$ if $\eps$ is small enough.

In section \ref{ana_F}, we simplify the left hand side (LHS) of \eqref{eq} by using a change of variables and recasting the result in terms of oscillatory integrals (OIs).  The right hand side (RHS) of \eqref{eq} is studied in section \ref{ana_Wj} where we employ a similar change of variables and show that these terms are small due to the $O(\e)$ hypothesis. Our uniqueness theorem follows directly in section \ref{unique}.

\subsection{Presentation in terms of OIs}\label{ana_F}
As in the linearized case, we will set $\xi_0=-\eta^\bot/|\eta|$.  In order to remove the resulting singularity at $\eta=0$ in our OI amplitudes, we first multiply by a cutoff function.  

Let $\chi_a(\eta)=\chi(|\eta|/a)\in C^\infty([0,\infty))$ be a cutoff function such that $\chi(x)=1$ for $x\ge 1$ and $\chi(x)=0$ for $0\le x\le 1/2$, where $a>0$ is a small parameter depending on $\eps$ to be chosen later. 
We denote
\begin{align}\label{Fm_int}
     \mathcal{F}_\pm \textbf{m}(\eta): = \int e^{-i\eta\cdot x_e}\chi_a(\eta)\m(x_1\pm \xi_1^\bot)dx_e,\ \ \pm=+,-,
\end{align}
then multiplying $\chi_a(\eta)$ on both sides of \eqref{eq} leads to
\begin{align}\label{eqcutoff}
    \mathcal{F}_\pm \textbf{m}(\eta)=\frac{1}{2}[\mathcal{W}^2\m\mp(\mathcal{W}^1\m+\mathcal{W}^3\m)^\bot],
\end{align} 
where 
$\mathcal{W}^j = \chi_a(\eta) W^j$.

Let us analyze the operators $\mathcal{F}_\pm$. We make the following changes of variables in the integral (\ref{Fm_int}) (recall \eqref{est:x} and \eqref{est:xi}):
\eqal{\label{cov}
x_e\to y&=x_1(x_e,\xi_0)\pm \xi_1^\bot(x_e,\xi_0)\\
&=x_e\pm\xi^\bot_0+O(\eps)\text{ in }C^{k-1},\\
x_e^\pm(y,\xi_0)&=\mp \xi^\bot_0+\phi_\pm(y,\xi_0),\\
&=\mp\xi^\bot_0+y+O(\eps)\text{ in }C^{k-1}.
}

The associated Jacobians are given by $J^{-1}_\pm(y,\xi_0)=1+O(\eps)$ in $C^{k-2}$, so we have
\begin{align}\label{Fpm}
   \mathcal{F}_\pm \textbf{m}(\eta) 
   &= \int e^{-i\eta\cdot x^\pm_e(y,\xi_0)}\m(y)J^{-1}_\pm(y,\xi_0)dy \notag\\
   &= \int e^{-i\eta\cdot \phi_\pm(y,\xi_\eta)}e^{\pm i|\eta|}\chi_a(\eta)\m(y)J^{-1}_\pm(y,\xi_\eta)dy,
\end{align}
where we can set 
$\xi_0$ to be
\eqal{\label{xi0}
\xi_\eta\equiv\xi_0=-\frac{\eta^\bot}{|\eta|}=\frac{1}{|\eta|}(-\eta_2,\eta_1),\qquad \eta\neq 0
}
due to $\eta\cdot\xi_0=0$.

Furthermore, the amplitudes of $\mathcal{F}_\pm$ can be described by the following symbol class.
\begin{definition}
 We say that $a(x,y,\xi)\in S^m_k$ if there exists a constant $C\geq 0$, such that 
 $$
      |\partial^\alpha_x\partial^\beta_y\partial^\gamma_\xi a(x,y,\xi)|\leq C|\xi|^{m-|\gamma|},\ \ |\alpha|+|\beta|+|\gamma|\leq k,
 $$
 for $x, \ y\in B_{\rho/2\pi}$ and $\xi\in \mathbb{R}^2\setminus \{0\}$. 
\end{definition}
It follows from \eqref{cov} and \eqref{Fpm} that $\mc F_\pm$ has the amplitude $\chi_a(\eta)J^{-1}_\pm(y,\xi_\eta)$ in $S^0_{k-2}$ , since
\eqal{\label{covs}
\phi_\pm(y,\xi_\eta)=y+O(\eps)\text{ in }S^0_{k-1},\\
J^{-1}_\pm(y,\xi_\eta)=1+O(\eps)\text{ in }S^0_{k-2}.
}

\subsection{The analysis of $\mathcal{W}^j$.}\label{ana_Wj}
Let us next characterize the integrals $\mathcal{W}^j\m=\chi_a(\eta)W^j\m$ for $j=1,2,3$.

Recall the definition of $W^j$ in \eqref{Dis}; then we have
\begin{align*} 
    \mathcal{W}^{j}\m (\eta)
    & =\int e^{-i\eta\cdot x_e}B^+_j(x_e,\xi_0)\m(x_1(x_e,\xi_0) + \xi^\bot_1(x_e,\xi_0))dx_e\\
    & \quad +\int e^{-i\eta\cdot x_e}B^-_j(x_e,\xi_0)\m(x_1(x_e,\xi_0) - \xi^\bot_1(x_e,\xi_0))dx_e,\ \  j=1,2,
\end{align*}
and 
\begin{align}\label{W3}
    \mathcal{W}^3\m (\eta) = -E(\xi_0)(\tau-\rho)\mathcal{W}^2\m(\eta)+2\pi E(\xi_0)\int e^{-i\eta\cdot x_e}(\xi_0\cdot x_e) \chi_a(\eta)\M_2dx_e,
\end{align}
where $B^\pm_j$ is a matrix function whose entries is a combination of entries of $B(s)$, such that $B^\pm_j$ is still $O(\eps)$ in $C^{k-2}$, see \eqref{Bs1}.

We first consider $\mathcal{W}^j$ for $j=1,2$ since they only contain small terms in their amplitudes. 
We apply the same changes of variables (\ref{cov}) to $\mathcal{W}^j$. Fixing $\xi_0$ as in \eqref{xi0}, we obtain
\begin{align*}
     \mathcal{W}^j\m(\eta)
     & =\int e^{-i\eta\cdot \phi_+(y,\xi_\eta)}e^{ i|\eta|}\chi_a(\eta)B^+_j(x_e(y,\xi_\eta),\xi_\eta)\m(y)J^{-1}_+(y,\xi_\eta)dy\\
     &\quad+\int e^{-i\eta\cdot \phi_-(y,\xi_\eta)}e^{ -i|\eta|}\chi_a(\eta)B^-_j(x_e(y,\xi_\eta),\xi_\eta)\m(y)J^{-1}_-(y,\xi_\eta)dy,\ \ j=1,2,
\end{align*}
where the amplitude matrices $\chi_a(\eta)B^\pm_jJ^{-1}_\pm$ in each integral are $O(\eps)$ in $S^0_{k-2}$.

Now we turn to operator $\mathcal{W}^3$. The first term of $\mathcal{W}^3 \m(\eta)$ is also $O(\varepsilon)$ in $S^0_{k-2}$ because $E(\xi_\eta)$ is homogeneous of degree 0 in $\eta$ and $\mathcal{W}^2\m(\eta)$ is $O(\varepsilon)$ in $S^0_{k-2}$. So we need only study the second term:
\begin{align}\label{W4}
    \mathcal{W}^4\m(\eta) := \int e^{-i\eta\cdot x_e}\,(\xi_0\cdot x_e)\,[\m(x_1+ \xi^\bot_1)+\m(x_1- \xi^\bot_1)]dx_e.
\end{align}
Applying the change of variables \eqref{cov} as before and noting that $\xi_\eta\cdot x_e^\pm=\xi_\eta\cdot\phi_\pm$, we get
\begin{align}\label{2ndterm}
    \mathcal{W}^4\m (\eta)
    =& \int e^{-i\eta\cdot x_e}\,(\xi_\eta\cdot x_e)\,[\m(x_1+ \xi^\bot_1)+\m(x_1- \xi^\bot_1)]dx_e  \notag\\
    = &\int \,e^{-i\eta\cdot \phi_+(y,\xi_\eta)+ i|\eta|}\,[\xi_\eta\cdot\phi_+(y,\xi_\eta)]\,\chi_a(\eta)\,\m(y)\,J^{-1}_+(y,\xi_\eta)dy  \notag\\
    & + \int \,e^{-i\eta\cdot \phi_-(y,\xi_\eta)- i|\eta|}\,[\xi_\eta\cdot\phi_-(y,\xi_\eta)]\,\chi_a(\eta)\,\m(y)\,J^{-1}_-(y,\xi_\eta)dy.
\end{align}

\begin{lemma}
    The operator $\mathcal{W}^4$ has amplitudes of $O(\eps/a)$ in $S^0_{k-3}$.
\end{lemma}
\begin{proof}
    From \eqref{eqcutoff}, 
    $$
    \mc F_+\m(\eta)+\mc F_-\m(\eta)=\mathcal{W}^2\m(\eta),
    $$
    so applying $\Xi:=\xi_\eta\cdot\nabla_\eta$ to this identity and noting that $\Xi(|\eta|)=0$ give
    \begin{align}\label{newid}
     \Xi\mathcal{W}^2\m(\eta)
     &= \int e^{-i\eta\cdot \phi_++ i|\eta|}[-i(\xi_\eta\cdot\phi_+ +\eta\cdot\Xi\phi_+)J^{-1}_++\Xi J^{-1}_+]\chi_a\m dy  \notag\\
     &\quad+\int e^{-i\eta\cdot \phi_-- i|\eta|}[-i(\xi_\eta\cdot\phi_- +\eta\cdot\Xi\phi_-)J^{-1}_-+\Xi J^{-1}_-]\chi_a\m dy.
    \end{align} 
    Therefore, from (\ref{newid}) and the definition \eqref{2ndterm}, it follows that
    \begin{align}\label{W4id}
        & \mathcal{W}^4\m(\eta)\notag\\
         &=\int e^{-i\eta\cdot\phi_++ i|\eta|}[(\xi_\eta\cdot\phi_+) B^+_2J^{-1}_++(\eta\cdot\Xi\phi_+) J^{-1}_+(B^{+}_2-1)+i\Xi((B^{+}_2-1)J^{-1}_+)]\chi_a\m dy  \notag\\
         &\quad +\int e^{-i\eta\cdot\phi_-- i|\eta|}[(\xi_\eta\cdot\phi_-) B^-_2J^{-1}_-+(\eta\cdot\Xi\phi_-) J^{-1}_-(B^{-}_2-1)+i\Xi((B^{-}_2-1)J^{-1}_-)]\chi_a\m dy \notag \\
         &=: \mathcal{W}^4_1 \m(\eta) + \mathcal{W}^4_2 \m(\eta).
    \end{align}
 
    The first term of $\mathcal{W}^4_k\m$, $k=1,2,$ is $O(\eps)$ in $S^0_{k-2}$ due to the small term $B^\pm_2$.  For the second term of $\mathcal{W}^4_k\m$, since 
    $$
        \eta\cdot \Xi\phi_\pm(y,\xi_\eta)=-\pd_{\xi^j}\phi^k_\pm(y,\xi_\eta)\frac{\eta^j\eta^k}{|\eta|^2}
    $$ 
    and 
    $ 
        \phi_\pm=y+O(\eps)
    $  in $C^{k-1}$, it follows that $(\eta\cdot\Xi\phi_\pm)J^{-1}_\pm(B^\pm_2-1)=O(\eps)$ in $S^0_{k-2}$. For the third term of $\mathcal{W}^4_k\m$, since 
    $$
        \Xi J^{-1}_\pm(y,\xi_\eta)=-\nabla_\xi J^{-1}_\pm(y,\xi_\eta)\cdot\frac{\eta}{|\eta|^2},
    $$ 
    and $J^{-1}_\pm=1+O(\eps)$ in $C^{k-2}$, we obtain that $\Xi J^{-1}_\pm=O(\eps)$ in $S^{-1}_{k-3}$, and similarly with  $B^\pm_2$.  However, upon multiplication by origin cutoff $\chi_a(\eta)=\chi(|\eta|/a)$, these $S^{-1}_{k-3}$ terms become ones that are only $O(\eps/a)$ in $S^0_{k-3}$. Since $S^0_{k-2}\subset S^0_{k-3}$, we conclude that $\mathcal{W}^4$ has amplitudes of $O(\eps/a)$ in $S^0_{k-3}$.

\end{proof}

Therefore, from definitions \eqref{W3} and \eqref{W4}, we can rewrite equation \eqref{eqcutoff} in the following form:
\begin{align}\label{eqf}
    \mc F_\pm\m(\eta)
   & = {1\over 2} \LB \mathcal{W}^2\m \mp (\mathcal{W}^1\m + \mathcal{W}^3\m)^\perp \RB \notag\\
   & = {1\over 2} \LB \mathcal{W}^2\m \mp (\mathcal{W}^1\m + (-E(\xi_0)(\tau-\rho) \mathcal{W}^2\m + 2\pi E(\xi_0) \mathcal{W}^4\m))^\perp \RB,\\
   \notag
   &=:\mc W^1_\pm\m+\mc W^2_\pm\m,
\end{align}
where
\begin{align*}
\mc W^1_\pm\m(\eta)&=\int e^{-i\eta\cdot\phi_+(y,\xi_\eta)+i|\eta|}W^1_\pm(y,\eta)\m(y)dy,\\
\mc W^2_\pm\m(\eta)&=\int e^{-i\eta\cdot\phi_-(y,\xi_\eta)-i|\eta|}W^2_\pm(y,\eta)\m(y)dy
\end{align*}
are OIs with amplitudes $W^j_\pm$ that are $O(\eps/a)$ in $S^0_{k-3}$. Here $W^j_\pm(y,\eta)$ is a combination of the symbols of the integrals on the second identity in \eqref{eqf} with phase $-i\eta\cdot\phi_\pm\pm i|\eta|$, respectively.
\subsection{$L^2$ estimates}\label{unique}
The following lemma is crucial for analyzing the $L^2$ boundedness of Fourier integral operators. %OIs and $\Psi$DOs.

  \begin{lemma}\label{opa}
  Let $r\leq0$, $l\in \mathbb{R}$. An operator $A$ is defined by
  $$
     Af(x)=\int\int e^{-i(x-y)\cdot\xi}a(x,y,\xi)f(y)dyd\xi .
  $$
  Suppose that 
  \begin{align*}
     \sum_{|\alpha|+|\beta|\leq 2n+[-r+|l|]+1}   \int\int |\partial^\alpha_x\partial^\beta_y a(x,y,\xi)| dxdy\leq M(1+|\xi|)^r,\ \hbox{for all}\ \xi\in \R^n,
  \end{align*}
  where $[s]$ is an integer such that $0\leq s-[s]< 1$. Then $A:H^{l+r}(\mathbb{R}^n)\rightarrow H^l(\mathbb{R}^n)$ is a bounded operator with the norm $\leq CM$ for some constant $C>0$, that is, 
  $$
      \|Af\|_{H^l(\mathbb{R}^n)}\leq CM\|f\|_{H^{l+r}(\mathbb{R}^n)}
  $$
  for all $f\in H^{l+r}(\mathbb{R}^n)$.
  \end{lemma}
    
 \begin{proof}
 When $r=l=0$, it was proved in the Appendix of \cite{SU1} by extending the proof of Theorem $18.1.11'$ in \cite{Hormander} to amplitudes $a(x,y,\xi)$. For the general case, that is, $r\leq0$, $l\in \mathbb{R}$, it was showed in \cite{W} by using a similar argument.
 \end{proof} 

\bigskip
We will study the $L^2$ norm on each side of (\ref{eqf}). For the left hand side, we let $\mc F^*_\pm$ be the adjoint of \eqref{Fpm} and consider the operators
$$
\mc F^*_\pm\mc F_\pm\m(w)=\int_{\re^2}\int_{\re^2}e^{i\eta\cdot(\phi_\pm(w,\xi_\eta)-\phi_\pm(y,\xi_\eta))}\chi_a^2(\eta)J^{-1}_\pm(y,\xi_\eta)J^{-1}_\pm(w,\xi_\eta)\m(y)\,dy\,d\eta,\quad \pm=+,-.
$$
As in \cite{SU3}, the phase function $\eta\cdot\phi_\pm$ admits the representation 
\begin{align*} 
        \eta\cdot(\phi_\pm(y,\xi_\eta)-\phi_\pm(x,\xi_\eta))=(y-x)\cdot \theta_\pm(x,y,\eta),
\end{align*}
where
\begin{align*}
        \theta_\pm(x,y,\eta)=\eta^j\int^1_0 (\nabla_x \phi^j_\pm)(x+t(y-x), \xi_\eta) dt.
\end{align*} 
Recall from \eqref{covs} that $\phi_\pm(y,\xi_\eta)=y+O(\eps)$ in $S^0_{k-1}$.  It follows that $\theta_\pm(x,y,\eta)=\eta+O(\eps)$ in $S^1_{k-2}$ and is homogeneous of degree one in $\eta$.  In particular,
\begin{align*}
    | \partial^\alpha_x\partial_y^\beta \partial^\gamma_\eta (\theta_\pm(x,y,\eta) - \eta)|
     \leq C\varepsilon|\eta|^{1-\gamma}\ ,|\alpha|+|\beta|\leq k-2,\ x,y\in B_{\rho/2\pi}.
\end{align*}
The equation $\theta = \theta_\pm(x,y,\eta)$ can be solved for $\eta$ if $\varepsilon$ is sufficiently small when $\eta\in \supp(\chi_a)	$ and $x,\ y \in B_{\rho/2\pi}$. Then the solution $\eta_\pm =\eta_\pm(x,y,\theta)$ satisfies $\eta_\pm = \theta + O(\varepsilon)$ in $S^1_{k-2}$. The corresponding Jacobian is
\begin{align*}
   J_{1,\pm}^{-1}(x,y,\theta)=| D \theta_\pm /D\eta_\pm |=1 + O(\varepsilon)\ \hbox{in $S^0_{k-3}$}.
\end{align*} 
We change variables, $\eta \rightarrow\theta$ in $\mc F_\pm^*\mc F_\pm\m(w)$ and obtain 
\begin{align*}
     \mc F_\pm^*\mc F_\pm\m(w)&=\int\int e^{i(w-y)\cdot\theta} 
        a_\pm( y, w, \theta ) \m(y) dy d\theta,
\end{align*}
where the amplitude is
\begin{align*}
       a_\pm( y,w,\theta ) = \chi_a^2( \eta_\pm(w,y,\theta) )  J^{-1}_\pm(y, \xi_{\eta_\pm}) J^{-1}_\pm(w, \xi_{\eta_\pm}) J_{1,\pm}(w,y,\theta). 
\end{align*}

To approximate $\mc F^*_\pm\mc F_\pm$, we define a new operator $\mc P$ by
$$
   \mc P\m(\theta) = \chi_a(\theta)\widehat{\m}(\theta).
$$
Then we get
 $$
   \mc P^*\mc P\m(w) = \int_{\re^2}\int_{\mathbb{R}^2} e^{i(w-y)\cdot \theta} \chi_a^2(\theta) \m(y)  dyd\theta.
 $$ 
Let $\psi$ be a smooth cut-off function supported in $B_{\rho/\pi}$ such that $\psi=1$ on $B_{\rho / 2\pi}$.  Then $\psi=1$ on $\supp( \m )$. We consider the operators
$$
     \psi (\mc F_\pm^*\mc F_\pm - \mc P^*\mc P)\psi,\qquad \pm=+,-,
$$
where the corresponding amplitudes
\begin{align*}
       & b_\pm(y,w,\theta ) = \psi(w) [\chi_a^2(\eta_\pm)   J^{-1}_\pm J^{-1}_\pm J_{1,\pm} - \chi_a^2(\theta) ] \psi(y)
\end{align*}
satisfy the following properties. 
\begin{lemma}\label{symbol} 
For $|\alpha|+|\beta|\leq k-3 $, we have
\begin{align}\label{b:est}
  |\partial^\alpha_y\partial_w^\beta \LC b_\pm(y, w, \theta) \RC | 
  \leq C\eps
\end{align}  
for all $y,\ w$ and $\theta\in \R^2$. Moreover, 
 \begin{align}\label{b:est1}
   \int \int|\partial^\alpha_y\partial_w^\beta \LC b_\pm(y, w, \theta) \RC | dydw
   \leq C\eps
 \end{align}
 for all $\theta\in \R^2$.
\end{lemma}

\begin{proof}
We first recall that $J_{1,\pm}=1+O(\varepsilon)$ in $S^0_{k-3}$, $J^{-1}_\pm=1+O(\varepsilon)$ in $S^0_{k-2}$, and $|\chi_a(\eta)|\le C$. Then $|b_\pm(y,w,\theta)|$ is small up to derivatives of order $k-3$. In addition, since $\psi$ is compactly supported, estimate (\ref{b:est}) holds for all $y,\ w$ and $\theta$ in $\R^2$, which leads to estimate (\ref{b:est1}).
\end{proof}

In order to apply Lemma \ref{opa} and Lemma \ref{symbol}, we require that the regularity $k$ satisfies $k-3 = 2\times 2 + 1$. Thus, $k = 8$ which explains the choice of regularity $Q_0^j\in C^8(\Omega)$ in Theorem \ref{q1}. We can now derive that 
 $$
      \|\psi(\mc F_\pm^*\mc F_\pm-\mc P^*\mc P)\psi\|_{L^2(\R^2)\rightarrow L^2(\R^2)}\leq C\varepsilon,\qquad\pm=+,-,
 $$
which implies
\begin{align*} 
      | \| \mc F_\pm \m \|^2_{L^2(\R^2)} -  \| \mc P \m \|^2_{L^2(\R^2)}|
     & = | \left< (\mc F_\pm^*\mc F_\pm-\mc P^*\mc P)\m, \m \right>| \\
     & = | \left< \psi(\mc F_\pm^*\mc F_\pm-\mc P^*\mc P)\psi\,\m, \m \right>| \\
     & \leq C\varepsilon \|\m\|^2_{L^2(\R^2)}.
\end{align*}

Therefore, we obtain the following lemma. 
 \begin{lemma}\label{opp}
  Let $g_j\in C^{k-1}_0(\R^2)$ with $k = 8$ for $j=1,2$, as defined in \eqref{appro1}.  Then we have
     $$
         \|\chi_a \widehat{\m}\|^2_{L^2(\R^2) }\leq   C\varepsilon\|\m\|^2_{L^2(\R^2)}+\|\mc F_\pm\m\|^2_{L^2(\R^2)}.
     $$
 \end{lemma}

Evaluating the right hand side of \eqref{eqf} yields an estimate of $\mc F_\pm\m$.  Let us apply a similar procedure as before to represent $(\mc W^j_{\pm})^*\mc W^j_{\pm}$ as an operator.  Since the amplitude matrix of $\mc W^j_{\pm}$ is $O(\eps/a)$ in $S^0_{k-3}$, it follows from Lemmas \ref{opa} and \ref{symbol} that 
$$
\|\psi(\mc W^j_{\pm})^*\mc W^j_{\pm}\psi\|_{L^2(\mathbb{R}^2)\rightarrow L^2(\mathbb{R}^2)}\le \frac{C\,\eps^2}{a^2},
$$
which implies $\|\mc W^j_{\pm}\m\|_{L^2(\mathbb{R}^2)}\le C\eps\|\m\|_{L^2(\R^2)}/a$. Hence, we obtain 
\eqal{\label{est:F}
    \|\mc F_\pm\m\|_{L^2(\re^2)}\le\frac{C\eps}{a}\|\m\|_{L^2(\re^2)}.
}

\textit{Proof of Theorem \ref{q1}:}
We combine Lemma \ref{opp} with (\ref{est:F}); then we acquire
   \begin{align*}
       \|\chi_a \widehat{\m} \|^2_{L^2(\R^2) }& \leq C\left(\eps+\frac{\varepsilon^2}{a^2}\right)\, \|\m\|^2_{L^2(\R^2)},
   \end{align*}  
that is,
\begin{align}
   \int \chi_a^2(\theta) |\widehat{\m}(\theta)|^2d\theta\leq C\left(\eps+\frac{\varepsilon^2}{a^2}\right)\|\m\|^2_{L^2(\R^2)},
\end{align}
where $\chi_a(\theta)=\chi(|\theta|/a)$ for $\chi\in C^\infty([0,\infty))$, with $\chi(x) = 1,x \geq 1$ and  $\chi(x)=0,x\le 1/2$.

By H\"older's inequality and the fact that $\supp(\m)\subset\Omega$, we have
\eqal{\label{o32}
     \|\m\|_{L^2(\R^2)}^2 
     &= \frac{1}{(2\pi)^2}\int_{|\theta|\leq a} |\widehat{\m}(\theta)|^2d\theta + \frac{1}{(2\pi)^2}\int_{|\theta|\geq a} \chi^2_a(\theta) |\widehat{\m}(\theta)|^2d\theta\\
     & \leq C a^2 \|\m\|_{L^2(\R^2)}^2 + C\left(\eps+\frac{\varepsilon^2}{a^2}\right) \|\m\|^2_{L^2(\R^2)},
}
where the first constant $C$ depends on $\Omega$. Since the $C$'s are independent of $a, \varepsilon$ and $\m$ provided that $\varepsilon$ is small enough,  
we %%
can %%
choose $a=\eps^{1/2}$ and $\varepsilon$ such that $Ca^2<1/2$ and $2C\eps<1/2$. Therefore,
$$
    \m = (\m_1, \m_2) = 0 \ \ \hbox{in $\Omega$}
$$
which implies that $Q_0^1 = Q_0^2$ in $\Om$ due to the fact that $Q_0^1 = Q_0^2$ in $\R^2 \backslash \Omega$. This completes the proof of Theorem \ref{q1}.

\section{A reconstruction formula}\label{sec:recons}
\label{sec:recon}

We consider the problem of explicitly reconstructing the background potential $Q_0$ from 
the 
scattering relation.
More precisely, we will give a reconstruction formula for $Q_0$ that is valid in the small $\eps$ limit.

We start with the Stefanov-Uhlmann identity (see \cite{CQUZ})
\begin{align}\label{SUid}
\int_0^\tau\frac{\pd X_e}{\pd X^{(0)}}(\tau-s,X(s,X^{(0)}))(V-V_e)(X(s,X^{(0)}))ds=X(\tau,X^{(0)})-X_e(\tau,X^{(0)}).
\end{align} 
The proof of this identity is as in Lemma \ref{lemmaid}, with $X_1$ and $X_2$ replaced by $X$ and $X_e$ respectively, and the  modification  $X(t(X^{(0)}),X^{(0)})=X_e(t(X^{(0)}),X^{(0)})$ is not required here. Let $\tau=\max\{\tau^1,\tau^2\}$, where $\tau^1$ and $\tau^2$ are the largest exit times of $X(s,X^{(0)})$ and $X_e(s,X^{(0)})$ respectively.  The exit times are defined as in \eqref{tau}, and from \eqref{e:zeroth0}, we have $\tau^2=2|\xi_0|\rho$.  Note that $\tau$ in \eqref{SUid} can be replaced with any larger time $T\ge \tau$.  For example, $T=4|\xi_0|\rho$ is sufficient as $\eps\to 0$.

Let us assume that 
$$
V-V_e=-\frac{1}{2}\begin{pmatrix}
\nabla^\bot Q_0(x+\xi^\bot)+\nabla^\bot Q_0(x-\xi^\bot)\\
\nabla Q_0(x+\xi^\bot)-\nabla Q_0(x-\xi^\bot)
\end{pmatrix}
$$
is small.  By linearizing identity \eqref{SUid} about the constant potential, we obtain the following approximate expression for the scattering data:
\eqal{\label{numeric_id}
 \int_0^\tau\frac{\pd X_e}{\pd X^{(0)}}(\tau-s,X_e(s,X^{(0)}))(V-V_e)(X_e(s,X^{(0)}))ds = \mathcal{S}(\theta(\xi_0^\bot),\al),
}
where we denote $\mathcal{S}(\theta,\al)=(\mathcal{S}_x,\mathcal{S}_\xi)=X(\tau,X^{(0)})-X_e(\tau,X^{(0)})$ that is a vector valued function in $\R^4$ with $ \theta(\xi_0^\bot)=\arctan(\xi^2_0,-\xi^1_0)$ for a given vector $\xi_0^\perp = (\xi_0^2, -\xi_0^1)$.  

\subsection{Derivation}
Since $\supp (g)\subset\Omega$, we can extend the integration interval to $\re$ %%on the
in identity \eqref{numeric_id}. Thus, we obtain the following integral equations
\begin{align}\label{numeric_id3}
\int_\re
\begin{pmatrix}
I_2&(\tau-s)E(\xi_0)\\
0&I_2
\end{pmatrix}
\begin{pmatrix}
g^\bot(x_e+\xi_0^\bot)-g^\bot(x_e-\xi_0^\bot)\\
g(x_e+\xi_0^\bot)+g(x_e-\xi_0^\bot)
\end{pmatrix}
ds=\mathcal{S}(\theta(\xi_0^\bot),\al).
\end{align}
Recall that $x_e=z-\frac{\rho}{2\pi}\frac{\xi_0}{|\xi_0|}+{s\over 2\pi} { \xi_0 \over |\xi_0|^2}$ with $z=\alpha{\xi^\perp_0\over|\xi_0|}$ for $\alpha\in\R $. Applying the Fourier transform in $\al$ and changing variables from $(\al,s)=(\frac{\xi_0^\bot}{|\xi_0|}\cdot x_e,\rho|\xi_0|+2\pi\xi_0\cdot x_e)$ to $x_e$ in \eqref{numeric_id3}, we obtain
$$
\int_{\re^2}e^{-i\beta\xi_0^\bot\cdot x_e}
\begin{pmatrix}
I_2&(\tau-s(x_e))E(\xi_0)\\
0&I_2
\end{pmatrix}
\begin{pmatrix}
g^\bot_+-g^\bot_-\\
g_++g_-
\end{pmatrix}
dx_e=\frac{1}{2\pi}\int_\re e^{-i\beta\al}\calS(\theta(\xi_0^\bot),\al)d\al,
$$
where we let $g_+ = g(x_e +\xi_0^\perp)$ and $g_- = g(x_e -\xi_0^\perp)$. Further, we rewrite the above identity by denoting $\beta=|\eta|$ for $\eta\in\re^2\setminus\{0\}$, and $\xi_0=-{\eta^\bot\over|\eta|}\equiv\xi_\eta$. Thus, we have
\begin{align}\label{6_id}
\int_{\re^2}e^{-i\eta\cdot x_e}
\begin{pmatrix}
I_2&(\tau-s(x_e))E(\xi_\eta)\\
0&I_2
\end{pmatrix}
\begin{pmatrix}
g^\bot_+-g^\bot_-\\
g_++g_-
\end{pmatrix}
dx_e=\widetilde\calS(\eta),
\end{align}
where the integrated scattering relation $\wt\calS=(\wt\calS_{x,1},\wt\calS_{x,2},\wt\calS_{\xi,1},\wt\calS_{\xi,2})$ is defined by 
\eqal{\label{scat}
\wt\calS(\eta)\equiv\frac{1}{2\pi}\int_\re e^{-i|\eta|\al}\calS(\theta(\eta),\al)d\al. 
}

We add and subtract the two equations from \eqref{6_id}; then we obtain
\begin{align*}
    2\int_{\re^2}e^{-i\eta\cdot x_e} g^\perp_+ dx_e+ \int_{\R^2}e^{-i\eta\cdot x_e} (\tau-s(x_e))E(\xi_\eta) (g_++g_-)dx_e = \wt{\mathcal{S}}_x+\wt{\mathcal{S}}^\perp_\xi,\\
     2\int_{\re^2}e^{-i\eta\cdot x_e} g^\perp_- dx_e -
     \int_{\R^2}e^{-i\eta\cdot x_e} (\tau-s(x_e))E(\xi_\eta) (g_++g_-)dx_e  = -\wt{\mathcal{S}}_x+\wt{\mathcal{S}}^\perp_\xi.
\end{align*} 
By the change 
%%changing of
of variables $y_\pm=x_e\pm\xi^\perp_\eta$, we further deduce that
\begin{align*}
    2e^{\pm i|\eta|}\int_{\re^2}e^{-i\eta\cdot y_\pm} g^\perp_\pm(y_\pm) dy_\pm \pm 2E(\xi_\eta)\cos(|\eta|) \int_{\R^2}e^{-i\eta\cdot y} (\tau-s(y))  g(y)dy = \pm\wt{\mathcal{S}}_x+\wt{\mathcal{S}}^\perp_\xi,
\end{align*} 
which, combined with the fact that $s(y) = \rho+2\pi\xi_\eta\cdot y$
and 
$$
   (\xi_\eta\cdot \nabla_\eta) \wh{g}(\eta) = -i \int_{\R^2} e^{-i\eta\cdot y} (\xi_\eta\cdot y) g(y)dy
$$ 
%%imply
implies
\begin{align*}
    2e^{\pm i|\eta|}\wh{g}_\pm^\perp(\eta) \pm 2 \cos(|\eta|) E(\xi_\eta) \LC  \tau-\rho 
    -2\pi i \xi_\eta\cdot \nabla_\eta \RC \wh{g}(\eta)  = \pm\wt{\mathcal{S}}_x+\wt{\mathcal{S}}^\perp_\xi.
\end{align*} 
Here and in what follows, we denote the Fourier transform of $f$ by $\wh{f}$.
From the above equations, we can obtain that
%Evaluating the left hand side of \eqref{6_id} and following a \textcolor{red}{similar argument in section \ref{sec:linear} } gives a differential system:
\begin{align}
\label{equ_1} 2i\sin(|\eta|)\,\wh g(\eta)^\bot+2\cos(|\eta|)\,E(\xi_\eta)(\tau-\rho-2\pi i\,\xi_\eta\cdot\nabla_\eta)\wh g(\eta)=\wt\calS_x(\eta),\\
\label{equ_2} 2\cos(|\eta|)\wh g(\eta)=\wt\calS_\xi(\eta),
\end{align}
where we view $\xi_\eta$ as a $2\times 1$ column vector such that $\xi_\eta\cdot \nabla_\eta = \xi_0^1 {\p\over \p\eta_1} + \xi_0^2 {\p\over \p\eta_2}$.

Using the identity
$$
    \cos(|\eta|)(\xi_\eta\cdot\nabla_\eta)\wh g(\eta)=(\xi_\eta\cdot\nabla_\eta)(\cos(|\eta|)\wh g(\eta)) 
$$
and substituting \eqref{equ_2} into \eqref{equ_1}, we obtain
\begin{align}
\label{sub1}
2i\sin(|\eta|)\,\wh g(\eta)^\bot&=\wt\calS_x(\eta)-E(\xi_\eta)(\tau-\rho-2\pi i\,\xi_\eta \cdot\nabla_\eta)\wt\calS_\xi(\eta),\\
\label{sub2}
2\cos(|\eta|)\,\wh g(\eta)&=\wt\calS_\xi(\eta).
\end{align}
 
\begin{lemma}
If equations \eqref{sub1}-\eqref{sub2} hold, then one has
\eqal{\label{idsimp}
2i\sin(|\eta|) \xi^\bot_\eta \cdot\wh g(\eta)=-\xi_\eta\cdot\wt\calS_x(\eta)+i|\eta|^{-1}\xi^\bot_\eta\cdot\wt\calS_\xi(\eta).
}
\end{lemma}

\begin{proof}
We first multiply \eqref{sub2} by the vector $\xi_\eta$; then we have	
\begin{align*}
\xi_\eta\cdot\wt\calS_\xi(\eta) = \xi_\eta \cdot (2\cos(|\eta|)\hat{g}(\eta)) = \xi_\eta \cdot \LC-\frac{1}{2}\,i\,\eta\,\wh Q_0(\eta)\RC=0,
\end{align*}
where we used the properties $\wh g(\eta)=-\frac{1}{2}\,i\,\eta\,\wh Q_0(\eta)$ and $\xi_\eta\cdot \eta=0$.
Furthermore, from the fact that $E(\xi_\eta)\xi_\eta = -{1\over 2\pi} \xi_\eta$, we get 
\begin{align}\label{dot}
    \xi_\eta \cdot \LC E(\xi_\eta) \calS_\xi(\eta)\RC = \LC   E(\xi_\eta)\xi_\eta\RC \cdot  \calS_\xi(\eta)=\LC -{1\over 2\pi}\xi_\eta\RC \cdot \widetilde\calS_\xi(\eta)= 0.
\end{align}   
Now we multiply both sides of \eqref{sub1} by $\xi_\eta$ and use \eqref{dot}; then we obtain
\begin{align}\label{S_id}
    2i\sin(|\eta|)\xi_\eta\cdot\wh g(\eta)^\bot=\xi_\eta\cdot\wt\calS_x(\eta) -i \xi_\eta\cdot (\xi_\eta\cdot\nabla_\eta) \wt\calS_\xi(\eta).
\end{align}
By a direct computation, we can get $\xi_\eta\cdot\wh g^\bot=-\xi_\eta^\bot\cdot\wh g$. To finish the proof, it remains to show that 
$$
      \xi_\eta\cdot (\xi_\eta\cdot\nabla_\eta) \wt\calS_\xi(\eta) =  {1\over|\eta|}\xi^\bot_\eta\cdot\wt\calS_\xi(\eta).
$$
We apply $\xi_\eta \cdot \nabla_\eta$ to the identity $\xi_\eta\cdot\ts_\xi(\eta) = 0 $, so that
\begin{align*}
     0=(\xi_\eta\cdot\nabla_\eta)(\xi_\eta\cdot\wt\calS_\xi(\eta))=\xi_\eta\cdot(\xi_\eta\cdot\nabla_\eta)\wt\calS_\xi(\eta)+ ((\xi_\eta\cdot\nabla_\eta)\xi_\eta) \cdot \wt\calS_\xi(\eta),
\end{align*}
which implies
\begin{align*} 
    \xi_\eta\cdot  (\xi_\eta\cdot\nabla_\eta)\wt\calS_\xi(\eta)
    &= - ((\xi_\eta\cdot\nabla_\eta) \xi_\eta) \cdot \wt\calS_\xi(\eta)\\
    &= {1\over |\eta|}\xi_\eta^\perp \cdot\wt\calS_\xi(\eta).
\end{align*}
Substituting it into the second term on the right hand side of \eqref{S_id}, this completes the proof.
\end{proof}

To solve for $\wh Q_0$, we first multiply \eqref{sub2} by $\xi_\eta^\bot$; then we have
$$
2\cos(|\eta|)\,\xi_\eta^\bot\cdot\wh g(\eta)=\xi_\eta^\bot\cdot\ts_\xi(\eta).
$$
Adding this to \eqref{idsimp} gives
\eqal{\label{recong}
2\xi_\eta^\bot\cdot \wh g(\eta)&=e^{-i|\eta|}\left[-\xi_\eta\cdot\wt{\mc S}_x(\eta)+\xi_\eta^\bot\cdot\ts_\xi(\eta)+i{1\over |\eta|}\xi_\eta^\bot\cdot\ts_\xi(\eta)\right].
}
Since $\xi_\eta^\bot=\frac{\eta}{|\eta|}$ and $\wh g(\eta)=-\frac{1}{2}\,i\,\eta\,\wh Q_0(\eta)$, we have
$$
2\xi_\eta^\bot\cdot \wh g(\eta)=-i \,|\eta|\,\wh Q_0(\eta).
$$
But $|\eta|\neq 0$, so combining this with \eqref{recong} yields
\eqal{\label{Qetanot0}
\wh Q_0(\eta)=i{e^{-i|\eta|}\over |\eta|}\left[-\xi_\eta\cdot\ts_x(\eta)+\xi_\eta^\bot\cdot\ts_\xi(\eta)+i{1\over |\eta|}\xi_\eta^\bot\cdot\ts_\xi(\eta)\right],\qquad \eta\in\re^2\setminus\{0\}.
}
We recall that $Q_0(x)=q=constant$ for $|x|\ge \rho/2\pi$.  This means $Q_0-q$ has compact support, such that $\wh Q_0(\eta)-(2\pi)^2\,q\,\delta(\eta)$ is an analytic function of $\eta$ well defined at $\eta=0$ (here, $\delta$ is the Dirac measure supported at zero). But for $\eta\neq 0$, this analytic function is \eqref{Qetanot0}, so the RHS of \eqref{Qetanot0} has a smooth extension to $\eta=0$.  We conclude that the Fourier transform of the potential is given by
\eqal{\label{Qhat0}
\wh Q_0(\eta)=(2\pi)^2\,q\,\delta(\eta) + i {e^{-i|\eta|}\over |\eta|} \left[-\xi_\eta\cdot\ts_x(\eta)+\xi_\eta^\bot\cdot\ts_\xi(\eta)+ i{1\over |\eta| }\xi_\eta^\bot\cdot\ts_\xi(\eta)\right],\qquad \eta\in\re^2.
}
 
In order to apply the inverse Fourier transform, we rewrite \eqref{Qhat0} in a more convenient form:
\eqal{\label{Qhat}
\wh Q_0(\eta) &=(2\pi)^2\,q\,\delta(\eta)-{e^{-i|\eta|}\over |\eta|^2}\left[\xi_\eta\cdot(i|\eta|\ts_x)-\xi_\eta^\bot\cdot(i|\eta|\ts_\xi)+\xi_\eta^\bot\cdot\ts_\xi(\eta)\right].
}
In polar coordinates $\eta=h(\cos\theta,\sin\theta)$, the second term is
$$
-{e^{-ih}\over h^2}\left[(-\sin\theta,\cos\theta)\cdot(i|\eta|\ts_x)-(\cos\theta,\sin\theta)\cdot(i|\eta|\ts_\xi)+(\cos\theta,\sin\theta)\cdot\ts_\xi(\eta)\right].
$$
Recalling integral transform \eqref{scat}, we invoke the following identities:
$$
i|\eta|\ts(\eta)=(\pd_\al\s)^\sim(\eta),\qquad f(\theta(\eta))\ts(\eta)=(f(\theta)\s)^\sim(\eta).
$$
Then \eqref{Qhat} becomes
\eqal{\label{Qhat1}
\wh Q_0(\eta)&=(2\pi)^2\,q\,\delta(\eta)-\frac{e^{-ih}}{h^2}\left[(-\sin\theta,\cos\theta)\cdot\pd_\al\s_x+(\cos\theta,\sin\theta)\cdot(\s_\xi-\pd_\al\s_\xi)\right]^\sim(\eta)\\
&=(2\pi)^2\,q\,\delta(\eta)-\frac{e^{-ih}}{h^2}\ts_0(\eta),
}
where we denote\eqal{\label{Delta0}
\s_0(\theta,\al):=(-\sin\theta,\cos\theta)\cdot\pd_\al\s_x(\theta,\al)+(\cos\theta,\sin\theta)\cdot(\s_\xi(\theta,\al)-\pd_\al\s_\xi(\theta,\al))
}
and $\ts_0 $ is defined as in \eqref{scat}.

%\Ru{Regularity on $Q_0$}
%To obtain $Q_0$, we apply the inverse Fourier transform.  If $Q_0-q\in C^8_c(\re^2)$, as in Theorem \ref{q1}, then the inverse transform is classically defined.  However, even if 
%%If $Q_0-q \in C^2_c(\re^2)$ such that $\s_0(\theta,\al)=\s_0(\theta(\xi_0^\bot),\xi_0^\bot\cdot x_0)$ is only $C_c(\mbb S^1,\re)$. In this setting, we can still interpret the inverse Fourier transform in the sense of distributions which allows us to obtain the reconstructions of less smooth potentials.  The calculations are given below.

In what follows, we let $\ms S$ be the space of Schwartz test functions on $\re^2$ and $(u,\vp)=\int_{\re^2}u(x)\vp(x)dx$ be the action of distribution $u$ on $\vp$.  We denote the Fourier transform and its inverse by $\ms F$ and $\ms F^{-1}$.
\begin{lemma}
Let $\vp\in\ms S$.  If potential $Q_0$ satisfies \eqref{Qhat1}, then
\eqal{\label{Qvp}
(Q_0-q,\vp)=\lim_{\eps\to 0}\frac{1}{(2\pi)^3}\int_{\re^2}\varphi(x)\int_0^{2\pi}\int_{\re}S_0(\theta,\al)J_\eps(1-r\cos(\theta-\phi),\al)\,d\al\, d\theta\, dx,
}
where $x=r(\cos\phi,\sin\phi)\in\re^2$, and for each $\eps>0$ and $\zeta,\al\in\re$, we let
\eqal{\label{Je}
J_\eps(\zeta,\al)=\int_0^\infty e^{-h(\eps+i\zeta)}\left(\frac{1-e^{-ih\al}}{h}\right)dh.
}
\end{lemma}

\begin{proof}
By the Fourier inversion formula for distributions, we have 
$$(Q_0-q,\vp)=(\ms F^{-1}\ms F(Q_0-q),\vp)=(\wh Q_0(\eta)-(2\pi)^2q\,\delta(\eta),\ms F^{-1}\vp(\eta)).$$  
Since $\ms F^{-1}\vp(\eta)$ is rapidly decreasing and $Q_0-q$ has a bounded Fourier transform, we can rewrite this as a limit using the dominated convergence theorem:
$$
(Q_0-q,\vp)=\lim_{\eps\to 0}\left(\wh Q_0(\eta)-(2\pi)^2q\,\delta(\eta),e^{-\eps|\eta|}\ms F^{-1}\vp(\eta)\right).
$$
Substituting in \eqref{Qhat1} and interchanging integrals, we get
\begin{align}\label{Qphi}
\begin{split}
(Q_0-q,\varphi)&=\lim_{\eps\to 0}-\frac{1}{(2\pi)^2}\int_{\re^2}\varphi(x)\int_{\re^2}\frac{e^{-\eps|\eta|+ix\cdot\eta-i|\eta|}}{|\eta|^2}\wt\s_0(\eta)d\eta dx\\
&=:\lim_{\eps\to 0}\frac{1}{(2\pi)^2}\int_{\re^2}\vp(x) I_\eps(x)dx,
\end{split}
\end{align}
where, in polar coordinates $\eta=h(\cos\theta,\sin\theta)$, $x=r(\cos\phi,\sin\phi)$, we denote
\eqal{\label{Ie}
I_\eps(x):=-\int_0^{2\pi}\int_0^\infty e^{-\eps h+ihr\cos(\theta-\phi)-ih}\frac{\wt\s_0(h,\theta)}{h}\,dh\,d\theta.
}

%Let us substitute \eqref{scat} into \eqref{Ie}.  In order for $\frac{e^{-ih}}{h^2}\wt\s_0(h,\theta)$ to be entire analytic, we must have $\wt\s_0(h,\theta)\to 0$ at $O(h^2)$ as $h\to 0$.  
In particular, from the definition \eqref{scat} and $\mathcal{S}(\theta,\alpha)$ is compactly supported in $\alpha$, we can deduce that	
\eqal{\label{intS0}
\wt\s_0(0,\theta)={1\over  2\pi} \int_\re\s_0(\theta,\al)d\al=0,\qquad\theta\in[0,2\pi).
}
Let $\chi\in C^\infty(\re)$ be a cutoff function with $\chi(x)=0$ for $x\le 1/2$ and $\chi(x)=1$ for $x\ge 1$.  From \eqref{intS0}, $I_\eps$ in \eqref{Ie} can be written as
\begin{align*}
I_\eps(x)&=\lim_{a\to 0}\int_0^{2\pi}\int_0^\infty e^{-\eps h+ihr\cos(\theta-\phi)-ih}\frac{\wt\s_0(0,\theta)-\wt\s_0(h,\theta)}{h}\,\chi\left(\frac{h}{a}\right)\,dh\,d\theta.
\end{align*}
Denote $\zeta=1-r\cos(\theta-\phi)$. By the definition \eqref{scat} and interchanging integrals, we get
\begin{align*}
I_\eps(x)&=\lim_{a\to 0}\frac{1}{2\pi}\int_0^{2\pi}\int_\re\int_0^\infty e^{-\eps h-ih\zeta}\left(\frac{1-e^{-ih\al}}{h}\right)\s_0(\theta,\al)\chi\left(\frac{h}{a}\right)dh\,d\al\, d\theta.
\end{align*}
By dominated convergence, the limit commutes with the integrals, and we obtain
\begin{align}\label{Ie1}
\begin{split}
I_\eps(x)&=\frac{1}{2\pi}\int_0^{2\pi}\int_\re \s_0(\theta,\al)\int_0^\infty e^{-\eps h-ih\zeta}\left(\frac{1-e^{-ih\al}}{h}\right)\,dh\,d\al\,d\theta\\
&=\frac{1}{2\pi}\int_0^{2\pi}\int_\re \s_0(\theta,\al)J_\eps(\zeta,\al)\,d\al\,d\theta,
\end{split}
\end{align}
where $J_\eps$ is given by \eqref{Je}.  Combined with \eqref{Qphi}, this recovers \eqref{Qvp}.
\end{proof}

We now turn to the evaluation of the integral \eqref{Qvp}. Then we have the following reconstruction formula for $Q_0$ valid in the small $\|Q_0\|$ limit.
%\begin{theorem}\label{recon_Q}
%In the linearized identity \eqref{numeric_id},  weak background potentials can then be reconstructed by using the formula
%\eqal{\label{recon0}
%Q_0(x)=q+\frac{1}{(2\pi)^3}\int_0^{2\pi}\int_{-\beta}^\beta \ln|\al+1-r\cos(\theta-\phi)|\,\s_0(\theta,\al)d\al d\theta,
%}
%where $\beta=1+\rho/2\pi$, $x=r(\cos\phi, \sin\phi)$ with $r\geq0$ and $0\leq\phi<2\pi$ in polar coordinates, and $\s_0$ is defined in \eqref{Delta0}.
%\end{theorem}

\begin{theorem}\label{recon_Q_1}
Let $q$ be a constant background potential. Suppose that the background potential $Q_0\in %%C^2(\Omega)$
C^8(\Omega)$ and $Q_0-q$ is compactly supported in $B_{\rho/2\pi}$. 
Then $Q_0$ can be reconstructed by using the formula
\eqal{\label{recon01}
Q_0(x)=q+\frac{1}{(2\pi)^3}\int_0^{2\pi}\int_{-\beta}^\beta \ln|\al+1-r\cos(\theta-\phi)|\,\s_0(\theta,\al)d\al d\theta,
}
where $\beta=1+\rho/2\pi$, $x=r(\cos\phi, \sin\phi)$ with $r\geq0$ and $0\leq\phi<2\pi$ in polar coordinates, and $\s_0$ is defined in \eqref{Delta0}.
\end{theorem}

\begin{proof} 
Let us choose any real valued $\vp\in\ms S$.  Since from \eqref{Qphi}
and $Q_0$ is a real valued function, we must have
\eqal{\label{real}
(Q_0-q,\vp)=\lim_{\eps\to 0}\frac{1}{(2\pi)^2}(\Re I_\eps,\vp),
}
where $\Re f$ denotes the real part of $f$, and $I_\eps$ is expressed in \eqref{Ie1}.  Therefore, it suffices to evaluate the real part of \eqref{Je}.  To do this, we first note that
$$
\Re\int_0^\infty e^{-h(\eps+i\zeta)}(1-e^{-ih\al})dh=\frac{\eps}{\eps^2+\zeta^2}-\frac{\eps}{\eps^2+(\zeta+\al)^2}=:K_\eps(\zeta,\al).
$$
Since $\frac{\p}{\p \eps}\Re J_\eps(\zeta,\al)=-K_\eps(\zeta,\al)$ and $\lim_{\eps\to\infty}J_\eps(\zeta,\al)=0$, we obtain:
$$
\Re J_\eps(\zeta,\al)=\int_\eps^\infty K_b(\zeta,\al)\,db=\frac{1}{2}\ln(\eps^2+(\zeta+\al)^2)-\frac{1}{2}\ln(\eps^2+\zeta^2).
$$
Let us substitute this into the real part of \eqref{Ie1}.  From \eqref{intS0}, the $\ln(\eps^2+\zeta^2)$ term integrates to zero, and we get
\eqal{\label{RIe}
\Re I_\eps(x)=\frac{1}{2\pi}\int_0^{2\pi}\int_\re\s_0(\theta,\al)\ln[(\eps^2+(\zeta+\al)^2)^{1/2}]\,d\al\,d\theta.
}
By dominated convergence, we can evaluate the $\eps\to 0$ limit in \eqref{real}:
\eqal{\label{recondist}
(Q_0-q,\varphi)=\frac{1}{(2\pi)^3}\left(\vp(x),\,\,\int_0^{2\pi}\int_\re\s_0(\theta,\al)\ln|\zeta+\al|\,d\al\,d\theta\right).
}
This holds for all real valued $\vp\in\ms S$.  To obtain \eqref{recon01} from \eqref{recondist}, we note that $\s(\theta,\al)=0$ if $|\al|\ge 1+\rho/2\pi$.  For these values of $\al$, neither vortex orbit intersects $B_{\rho/2\pi}$, so $X(\tau,X^{(0)})-X_e(\tau,X^{(0)})=0$.
\end{proof}

\subsection{Rescaled initial conditions}
\label{subsec:rescale}
We now present the reconstruction formula using more flexible initial conditions.  To obtain formula \eqref{recon01}, we assumed $|\xi_0|=1$.  We now relax this convention and fix $\sigma\in(0,\infty)$.  We intend to take $|\xi_0|=\sigma$.  In order to relate the previous situation of $|\xi_0|=1$ to the more general case, we use the following scaling covariance of Hamiltonian systems \eqref{e:xeqn}, \eqref{e:xieqn}:
\eqal{
x(s)&\to \wt x(\wt s)=\sigma x(s),\\
\xi&\to\wt \xi(\wt s)=\sigma \xi(s),\\
s&\to \wt s=\sigma^2\,s,\\
Q_0(y)&\to \wt Q_0(y)=Q_0\left(\frac{y}{\sg}\right).
}
In other words, if $(x(s),\xi(s))$ is a solution for potential $Q_0(y)$, the rescaled vector $(\wt x(s),\wt\xi(s))=(\sigma x(\frac{s}{\sg^2}),\sg\xi(\frac{s}{\sg^2}))$ is a solution for rescaled potential $\wt Q(y)=Q_0(\frac{y}{\sg})$.

In order for the initial conditions $(x(0),\xi(0))=(\al\frac{\xi_0^\bot}{|\xi_0|}-\frac{\rho}{2\pi}\frac{\xi_0}{|\xi_0|},\xi_0)$ to rescale in this way, or
$$
(x_0,\xi_0)\to(\wt x_0,\wt\xi_0)=(\sg x_0,\sg\xi_0),
$$
we need for the impact parameter $\al$ and support radius $\rho$ to rescale as follows:
$$
(\al,\rho)\to(\wt\al,\wt\rho)=(\sg\al,\sg\rho).
$$
The exit time $\tau$ in \eqref{tau} rescales like $s$: $\tau\to\wt\tau=\sg^2\tau$.  The scattering relation $\s=X(\tau,X^{(0)})-X_e(\tau,X^{(0)})$ becomes
$$
\s(\theta,\al)\to\wt\s(\theta,\wt\al)=\sg\s(\theta,\al).
$$
This means $\s(\theta,\al)=\frac{1}{\sg}\wt\s(\theta,\sg\al)$.  Substituting this into \eqref{Delta0}:
\begin{align*}
\s_0(\theta,\al)&=\frac{1}{\sg}\left[(-\sin\theta,\cos\theta)\cdot\,\sg\pd_u\wt\s_x(\theta,u)+(\cos\theta,\sin\theta)\cdot(\wt\s_\xi(\theta,u)-\sg\pd_u\wt\s_\xi(\theta,u))\right]\Big|_{u=\sg\al}\\
&=:\frac{1}{\sg}\wt\s_0(\theta,\wt\al).
\end{align*}

Let us substitute this information into \eqref{recon01}, i.e.
\begin{align*}
x=\frac{\wt x}{\sg},\qquad Q_0(x)=\wt Q_0(\wt x),\qquad\rho=\frac{\wt\rho}{\sg},\qquad\s_0(\theta,\al)=\frac{1}{\sg}\wt\s_0(\theta,\sg\al).
\end{align*}
We get
\begin{align*}
\wt Q_0(\wt x)-q&=\frac{1}{(2\pi)^3\sg}\int_0^{2\pi}\int_{-\beta}^\beta\ln|\al+1-\frac{\wt r}{\sg}\cos(\theta-\wt\phi)|\wt\s_0(\theta,\sg\al)d\al d\theta\\
&=\frac{1}{(2\pi)^3\sg^2}\int_0^{2\pi}\int_{-\wt\beta}^{\wt\beta}\ln|\wt\al+\sg-\wt r\cos(\theta-\wt\phi)\wt\s_0(\theta,\wt\al)d\wt\al d\theta,
\end{align*}
where $\wt\beta:=\sg\beta=\sg+\wt\rho/2\pi$, and we used \eqref{intS0}.

\bigskip
Therefore, the generalized reconstruction formula for $|\xi_0|=\sg\in(0,\infty)$ is as follows:
\eqal{\label{recon}
Q_0(x)=q+\frac{1}{(2\pi)^3\,\sigma^2}\int_0^{2\pi}\int_{-\beta}^{\beta} \,\ln|\al+\sigma-r\cos(\theta-\phi)|\,\s_0(\theta,\al)d\al d\theta,
}
where $\beta\equiv\sigma+\rho/2\pi$, $x=r(\cos\phi,\sin\phi)$, and
\begin{align}\label{S0}
\s_0(\theta,\al)&\equiv\sigma(-\sin\theta,\cos\theta)\cdot\pd_\al\s_x(\theta,\al)+(\cos\theta,\sin\theta)\cdot(\s_\xi(\theta,\al)-\sigma\pd_\al\s_\xi(\theta,\al)).
\end{align}
Setting $\sg=1$ recovers \eqref{Delta0} and \eqref{recon01}.

\subsection{Radial potentials}
As a distinguished special case, let us consider radial (rotationally invariant) potentials $Q_0(x)=Q(|x|)$, for which the reconstruction formula simplifies considerably.  For these potentials, the Hamiltonian systems \eqref{e:xeqn}, \eqref{e:xieqn} admit the following rotation symmetry:
$$
(x(s),\xi(s))\to(x^{(\theta)}(s),\xi^{(\theta)}(s))=(R_\theta\,x(s), R_\theta\,\xi(s)),
$$
where $R_\theta=\begin{pmatrix}
\cos\theta&-\sin\theta\\
\sin\theta&\cos\theta
\end{pmatrix}$ is a rotation matrix, for each $\theta\in\re$ (we are interpreting $x$ and $\xi$ as column vectors).  In other words, for each $Q(|x|)$, if $(x,\xi)$ is the solution with initial conditions $(x_0,\xi_0)$, then $(x^{(\theta)},\xi^{(\theta)})=(R_\theta\,x,R_\theta\,\xi)$ is the solution with initial conditions $(R_\theta\,x_0,R_\theta\,\xi_0)$.

\medskip
Recall that scattering relation $\s(\theta,\al)$ was defined for initial conditions $(x_0,\xi_0)=(\al\frac{\xi_0^\bot}{|\xi_0|}-\frac{\rho}{2\pi}\frac{\xi_0}{|\xi_0|},\xi_0)$, where $\xi_0^\bot=\sg(\cos\theta,\sin\theta)$ and $\xi_0=\sg(-\sin\theta,\cos\theta)$ for each $(\theta,\al)$.  We will denote these initial conditions by $(x^{(\theta)}_0,\xi^{(\theta)}_0)$.  Then it is easy to check that \textcolor{red}{
}
$$
(x^{(\theta)}_0,\xi^{(\theta)}_0)=(R_\theta\,x^{(0)}_0,R_\theta\,\xi^{(0)}_0) 
$$
for all $\theta\in\R$, where
$$
    x^{(0)}_0=
    \begin{pmatrix}
    \al\\
    -\frac{\rho}{2\pi}
    \end{pmatrix},
    \qquad\xi^{(0)}_0=
    \begin{pmatrix}
    0\\
    \sg
    \end{pmatrix}.
$$ 
Furthermore, by rotational symmetry, we have
$$
(x^{(\theta)}(s),\xi^{(\theta)}(s))=(R_\theta\,x^{(0)}(s),R_\theta\,\xi^{(0)}(s)),
$$
so the scattering relation $\s=X(\tau,X^{(0)})-X_e(\tau,X^{(0)})$ varies with $\theta$ as follows:
\eqal{\label{rot}
\s_y(\theta,\al)=
R_\theta\,\s_y(0,\al),\qquad y=x,\ \xi.
}
If we substitute these relations into \eqref{S0}, we find that the $\theta$ dependence in $\s_0$ disappears:
\eqal{\label{s0r}
\s_0(\theta,\al)&=\sigma\pd_\al\s_{x,2}(0,\al)+\s_{\xi,1}(0,\al)-\sigma\pd_\al\s_{\xi,1}(0,\al)\\
&=:\s_0(\al),
}
where $\s_x=(\s_{x,1} , \s_{x,2})$ and $\s_\xi=(\s_{\xi,1} , \s_{\xi,2})$.
Therefore, the $\theta$ integral in reconstruction formula \eqref{recon} can be performed explicitly.  The resulting formula for a radial potential is a single integral:
\eqal{\label{reconr}
Q(r)=q+\frac{1}{(2\pi\sigma)^2}\int_{-\beta}^\beta \LC\ln\,\lambda(r,\al)\RC\s_0(\al)d\al,
}
where $\lambda(r,\al)$ is given away from the point $(r,\al)=(0,-\sigma)$ by
\eqal{\label{lam}
\lambda(r,\al)=
\begin{cases}
\frac{1}{2}r,&r>0\text{ and } |\al+\sigma|\le r,\\
\frac{1}{2}(|\al+\sigma|+\sqrt{(\al+\sigma)^2-r^2}),&0\le r< |\al+\sigma|.
\end{cases}
}
This $\theta$ integral was computed with a symbolic integration program, Mathematica 9.0.

\medskip
For radial potentials reconstructed using \eqref{reconr}, we see that although the approximately reconstructed $Q_0-q$ will have compact support as required, the radius of the support will be too large by a factor of $2\sigma$, the dipole distance.  Indeed, if $r\ge \beta+\sigma$, then the integral kernel in \eqref{reconr} is 
$$
\frac{1}{(2\pi\sigma)^2}\ln(r/2)
$$
for all $\al\in[-\beta,\beta]$.  Since $\int_\R \s_0(\al)d\al=0$ by \eqref{intS0}, this implies that $Q(r)=q$ for $r\ge\beta+\sigma=\rho/2\pi+2\sigma$.  Since the same cannot be said for $r< \beta+\sigma$, $Q_0-q$ may not be zero in this ball.  This support increase can also be understood from the exponential order of the Fourier transform \eqref{Qhat1}.

\section{Numerical Reconstructions}
\label{sec:numer}
Given %%
the
scattering relation, evaluating one of integrals \eqref{recon} or \eqref{reconr} is sufficient to reconstruct weak background potentials.  Such numerical integration is straightforward, and the computation times for \eqref{reconr} are, at most, a few seconds using basic software.  For stronger potentials, an iteration method such as that in \cite{CQUZ} might be developed.

We first reconstruct a radial potential with compact support:
\eqal{\label{poly}
Q_0(x)=\eps\,\left(1-\frac{|x|^2}{\omega^2}\right)^{\kappa+1},\qquad 0\le |x|\le \omega,
}
with $Q_0(x)=0$ for other $x$.  Here, $\eps>0$ is a small parameter, $\kappa>0$ controls the smoothness of $Q_0$ (i.e. $Q_0\in C^{\kappa}(\re^2)$), and $\omega>0$ is the support radius of $\nabla Q_0$.  

To generate %%
the 
scattering relation for various potential strength $\epsilon$, we choose $\rho/2\pi=1$, $\omega=1/2$, $\sigma=1/10$, $\kappa=8$, and $\tau=2\sigma\rho$.  We solve ODE system \eqref{odesys} numerically, using a differential equation solver ``NDSolve" in Mathematica 9.0, for the range of initial conditions $\xi_0=(0,\sigma)$ and $x_0=(\al_\ell,-\frac{\rho}{2\pi})$, where $\al_\ell=\beta(\frac{\ell}{N}),\ell=0,\pm1,\dots,\pm N$, with $\beta=\sigma+\rho/2\pi$ and $N=400$, say.  Evaluating $\s(0,\al)=X(\tau,x_0,\xi_0)-X_e(\tau,x_0,\xi_0)$ yields a table of scattering relation, plotted in Figure \ref{fig:scat} for $\eps=0.01$.

To evaluate integral \eqref{reconr} numerically for several $r\in[0,\beta]$, we used the composite Simpson's rule, where $[-\beta,\beta]$ was again discretized according to $\al_\ell=\beta(\frac{\ell}{N})$.  Function $\s_0(\al)$ in \eqref{s0r} was computed at the points $\al_\ell$ using central difference quotients, $\pd_\al\s(\al_\ell)\approx\frac{\s(\al_{\ell+1})-\s(\al_{\ell-1})}{\al_{\ell+1}-\al_{\ell-1}}$, with $\s(\al_{\ell})=0$ if $|\ell|\ge N+1$.  See Figure \ref{fig:S0}.

\begin{figure}
   \centering
   \includegraphics[width=4in]{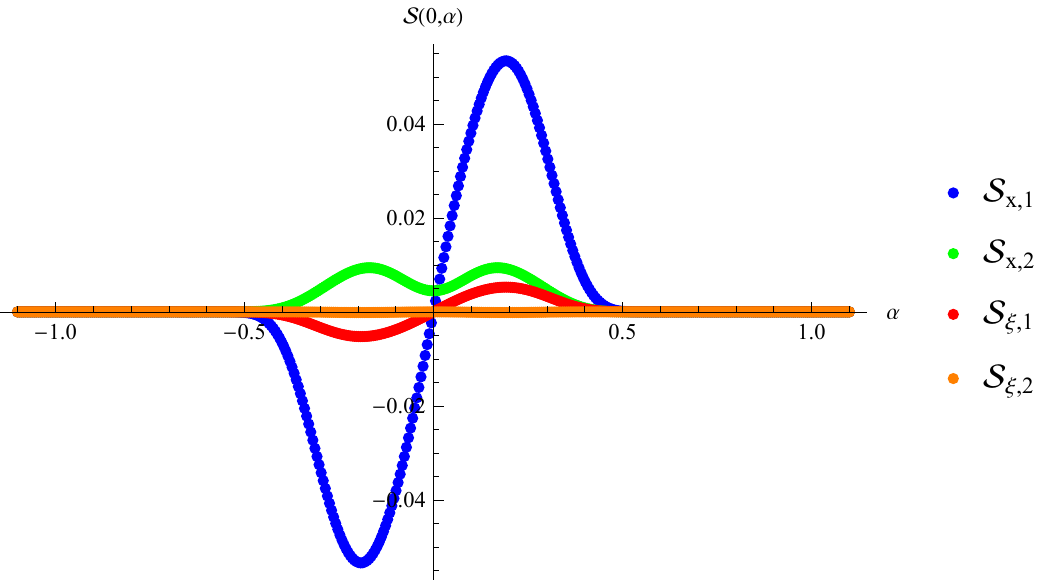} 
   \caption{Numerically generated scattering relation from \eqref{poly} for $\eps=0.01$, $\omega=0.5$, and $\kappa=8$.}
   \label{fig:scat}
\end{figure}

\begin{figure}
   \centering
   \includegraphics[width=3.5in]{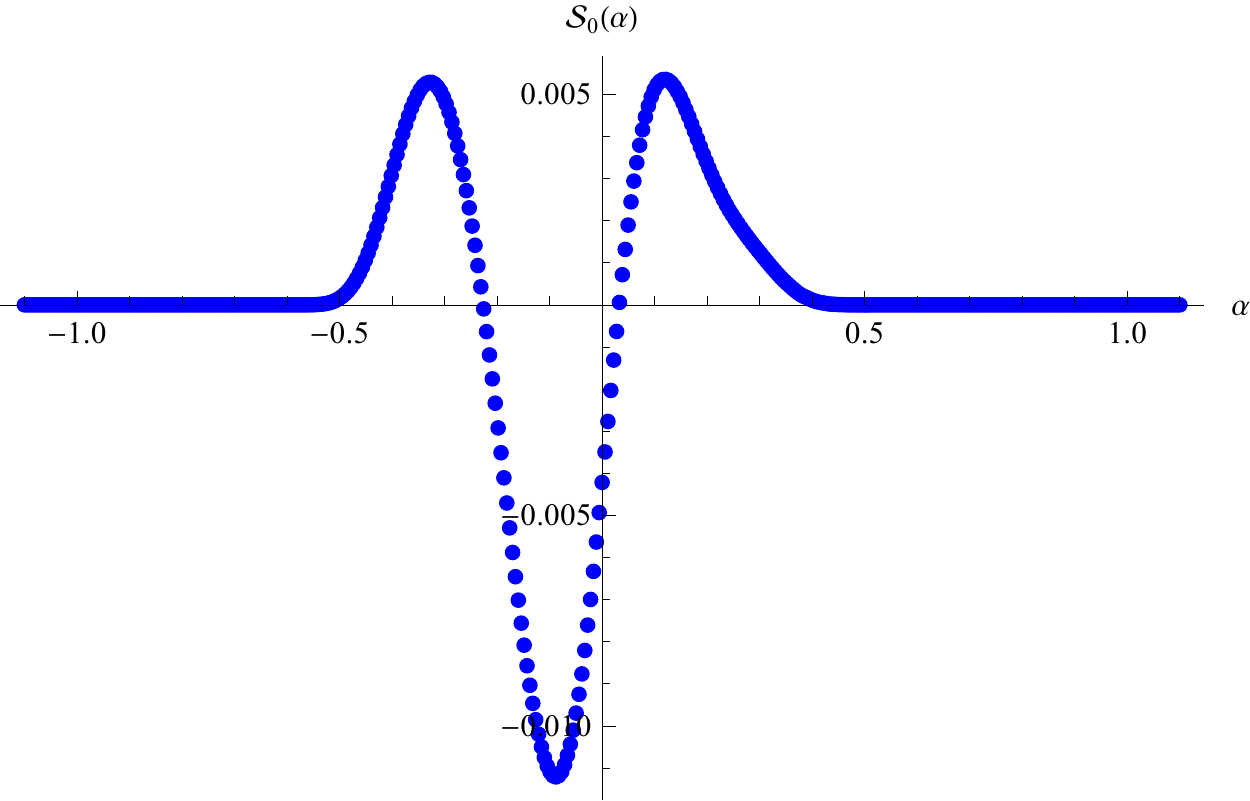} 
   \caption{Scattering function $\s_0(\al)$ in \eqref{s0r} from \eqref{poly} for $\eps=0.01$, $\omega=0.5$, and $\kappa=8$.}
   \label{fig:S0}
\end{figure}

In Figure \ref{fig:poly}, the reconstructions (dashed lines) are compared with the exact potential \eqref{poly} (solid line).  Here, $Q(r)/\eps$ is plotted for various $\eps$.  It is clear that the reconstructions improve as $\eps$ decreases.  For larger $\eps$, the agreement is poor, which indicates that the linearization \eqref{numeric_id} becomes invalid.

% \begin{comment}
\begin{figure}
   \centering
   \includegraphics[width=4in]{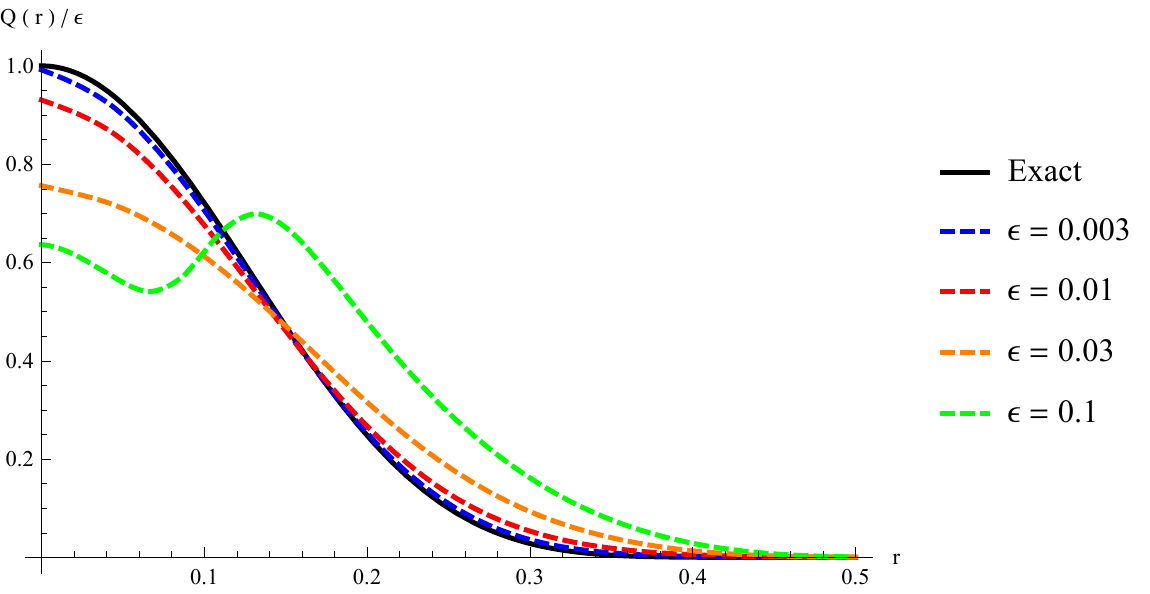} 
   \caption{Reconstruction of differentiable potential with compact support \eqref{poly} (solid line) using \eqref{reconr} for several $\eps$ (dashed line) and $\kappa=8$.}
   \label{fig:poly}
\end{figure}
% \end{comment}

\medskip
If we choose $\kappa=0$ in \eqref{poly}, then $\nabla Q_0$ is not continuous ($Q(r)$ has a cusp at $r=\omega$). However, reconstructions of weak background potentials are still possible.  Letting the other parameters be as before, we present reconstructions in Figure \ref{fig:poly_k1} for various $\eps$.  The error still vanishes uniformly as $\eps\to0$.  

%\begin{comment} 
\begin{figure}
   \centering
   \includegraphics[width=4in]{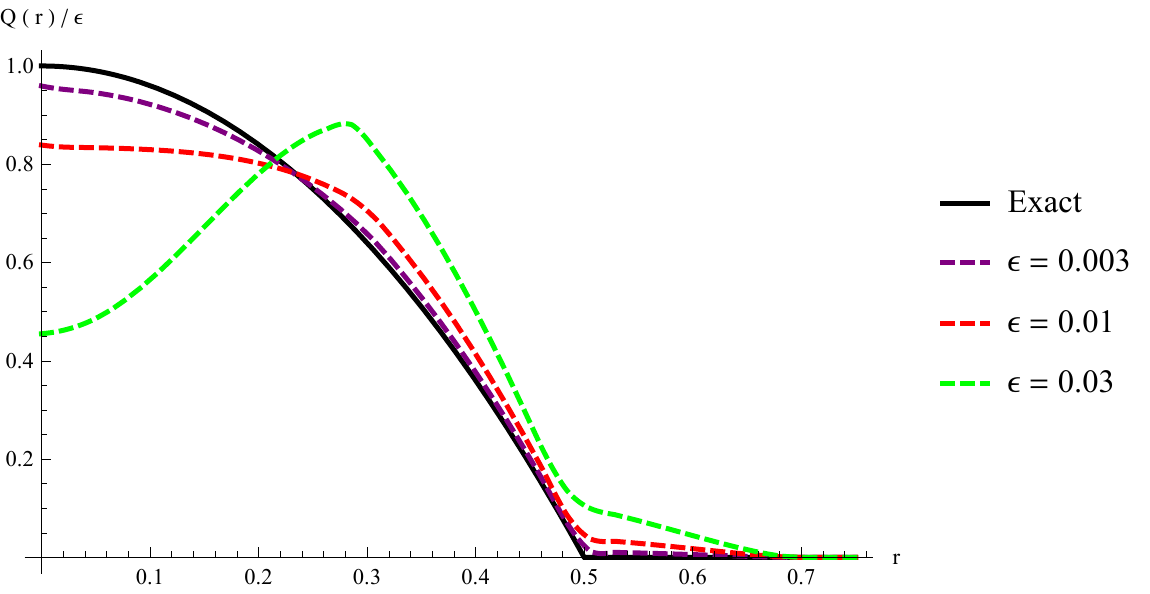} 
   \caption{Reconstruction of compactly supported potential with cusp \eqref{poly} (solid line) using \eqref{reconr} for several $\eps$ (dashed line) and $\kappa=0$.}
   \label{fig:poly_k1}
\end{figure}
% \end{comment}

Figure \ref{fig:poly_k1} also clearly demonstrates how the reconstruction adds support.  Although $Q_0$ in \eqref{poly} is supported on $|x|\le 0.5$, the reconstructions are supported on $|x|\le 0.7$.  The additional $0.2$ support radius corresponds to the dipole distance $2\sigma$.  As a consequence, the approximations appear more smooth at $r=0.5$ (and less near $r=0.7$) than the exact potential.

\medskip
We also use \eqref{reconr} to reconstruct the following potential that does not have compact support:
\eqal{\label{gauss}
Q_0(x)=\eps\,e^{-10\left(\frac{|x|}{\rho/2\pi}\right)^2}.
}
Our reconstruction aims to approximate $Q_0$ only in $B_{\rho/2\pi}$; because we extend $\s_0(\al)=0$ for $|\al|\ge \beta$, the resulting reconstruction will have compact support.  Letting parameters be as before, the reconstructions are presented in Figure \ref{fig:Gauss}.  Those with small $\eps$ clearly recover the potential in the region $r\le 1$.

% \begin{comment}
\begin{figure}
   \centering
   \includegraphics[width=4in]{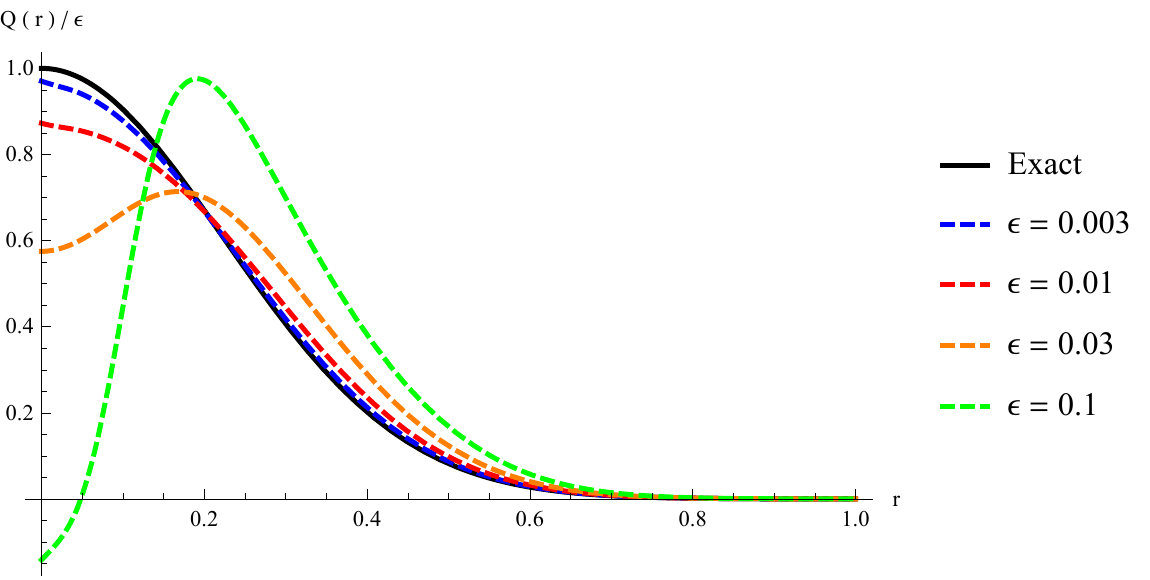} 
   \caption{Reconstruction of Gaussian potential \eqref{gauss} (solid line) using \eqref{reconr} for several $\eps$ (dashed line).}
   \label{fig:Gauss}
\end{figure}
% \end{comment

% % % % % % % % % % % % % % % % % % % % % % % % % % % % % % % % % % % % % %

% % % % % % % % % % % % % % % % % % % % % % % % % % % % % % % % % % % % % %

\bibliographystyle{abbrv}

\bibliography{dipolebib}

 \end{document}